\documentclass[a4paper,12pt,reqno]{amsart}
\usepackage{latexsym, amsfonts, amsmath, amsthm, amssymb, enumerate, verbatim}
\usepackage[top=3cm, bottom=3cm, left=2.54cm, right=2.54cm]{geometry}

\renewcommand\Re{\mathop{{\rm Re}}}
\renewcommand\Im{\mathop{{\rm Im}}}

\newcommand{\HT}{\mathcal{H}}
\newcommand{\Log}{\mathcal{L}}
\newcommand{\ds}{\displaystyle}
\newcommand{\pv}{\mathop{\mathrm{p.v.}}}
\newcommand{\sgn}{\mathop{\mathrm{sgn}}}
\newcommand{\st}{\,:\,}
\newcommand{\sumd}{\sideset{}{'}\sum}

\newtheorem{thm}{Theorem}[section]
\newtheorem*{thm*}{Theorem}
\newtheorem{cor}[thm]{Corollary}
\newtheorem*{cor*}{Corollary}
\newtheorem{lem}[thm]{Lemma}
\newtheorem{prop}[thm]{Proposition}

\newtheorem*{con*}{Conjecture}
\newtheorem*{prob*}{Problem}

\theoremstyle{definition}
\newtheorem{defn}[thm]{Definition}
\theoremstyle{remark}
\newtheorem{rem}[thm]{Remark}

\numberwithin{equation}{section}

\begin{document}

\title{Statistical mechanics of the periodic Benjamin--Ono equation}
\author{Gordon Blower}
\email{g.blower@lancaster.ac.uk}
\address{Department of Mathematics and Statistics, Lancaster
University, Lancaster LA1 4YF, United Kingdom}

\author{Caroline Brett}
\email{caroline.brett@gssi.it}
\address{Gran Sasso Science Institute, Viale Francesco Crispi n.7, 67100 L'Aquila AQ, Italy}

\author[Ian Doust]{Ian Doust}
\email{i.doust@unsw.edu.au}
\address{School of Mathematics and Statistics, University of New South Wales, Sydney, NSW 2052, Australia}

\subjclass[2010]{}

\date{1st February 2019}

\begin{abstract} The periodic Benjamin--Ono equation is an autonomous Hamiltonian system with a Gibbs measure on $L^2({\mathbb T})$. The paper shows that the Gibbs measures on bounded balls of $L^2$ satisfy some logarithmic Sobolev inequalities. The space of $n$-soliton solutions of the periodic Benjamin--Ono equation, as discovered by Case, is a Hamiltonian system with an invariant Gibbs measure. As $n\rightarrow\infty$, these Gibbs measures exhibit a concentration of measure phenomenon. Case introduced soliton solutions that are parameterised by atomic measures in the complex plane. The limiting distributions of these measures gives the density of a compressible gas that satisfies the isentropic Euler equations.
\end{abstract}

\maketitle

\section{Introduction}\label{S:Intro}

The Benjamin--Ono equation is an integro-differential equation which was originally introduced in the study of waves in deep water \cite{Ben1967}. Conceptually, the Benjamin--Ono equation lies between the Burgers equation and the KdV equation, hence shares some of the properties of the latter differential equation. More recently, various authors have studied the periodic version of the equation, proving for example, theorems on global well-posedness (see, for example, \cite{Mol2007}).

The Benjamin--Ono equation exhibits technical challenges which are not present in the case of the NLS and KdV equations in one space dimension. Specifically, the presence of the Hilbert transform means that solutions can be influenced by behaviour which is simultaneous but at a large distance.

%
Let ${\mathbb T}={\mathbb R}/2\pi {\mathbb Z}$ be the unit circle. The Hilbert transform on ${\mathbb T}$ is the operator $\HT : L^2({\mathbb T}; {\mathbb C})\rightarrow L^2({\mathbb T}; {\mathbb C})$
  \[ \HT v(x)= \pv \int_{\mathbb T} \cot \left({\frac{x-y}{2}}\right) v(y)\, \frac{dy}{2\pi } \qquad (x \in {\mathbb T}). \]
The Hilbert transform may be expressed as the Fourier multiplier $\HT :e^{inx}\mapsto -i \sgn(n) e^{inx}$, where we take $\sgn(0) = 0$.

The periodic Benjamin--Ono equation with real parameter $\beta$ is the evolution equation,
\begin{equation}\label{BO-eqn}
  \frac{\partial u}{\partial t} +\HT \frac{\partial ^2u}{\partial x^2} +2\beta \frac{\partial u}{\partial x} u=0
\end{equation}
where $u = u(x,t):{\mathbb T}\times {\mathbb R}\rightarrow {\mathbb R}$ is twice continuously differentiable, and $\mathcal{H}$ is acting on the space variable $x$.

The Hamiltonian
  \begin{equation}\label{Hamilt}
   H_\beta (u)=\frac{1}{2} \int_{\mathbb T} \HT \frac{\partial u}{\partial x} (x,t) \, u(x,t)\, \frac{dx}{2\pi}
        +\frac{\beta}{3} \int_{\mathbb T} u(x,t)^3\, \frac{dx}{2\pi},
  \end{equation}
has canonical equation of motion
  \begin{equation}\label{CEM}
   \frac{\partial u}{\partial t} = \frac{\partial }{\partial x} \frac{\partial H_\beta}{\partial u}
  \end{equation}
which gives rise to (\ref{BO-eqn}). Under the evolution (\ref{BO-eqn}), $H_\beta (u)$ is invariant with respect to time $t$, as is
   \begin{equation}\label{N-number}
   N(u)=\int_{\mathbb T} u(x,t)^2 \, \frac{dx}{2\pi}.
   \end{equation}

The periodic Benjamin-Ono equation is this autonomous Hamiltonian system, which can be viewed as the limit of a sequence of
autonomous Hamiltonian systems that have phase spaces modelled on the finite-dimensional vector spaces that are spanned by the first $M$ modes of the trigonometric basis of $L^2({\mathbb T})$. For each such system, the canonical equations of motion give a system of ordinary differential equations that has a Gibbs measure which is invariant by the classical Liouville theorem.  Hence it is natural to regard the limit of these finite dimensional Gibbs measures as the Gibbs measure for the Benjamin--Ono system itself.  Deng, Tzvetkov and Visciglia \cite{DTV2015} constructed such an invariant measure for (\ref{BO-eqn}), which is absolutely continuous with respect to the  free measure for $\beta=0$, and for which the  initial value problem is well-posed on the support of the measure. Furthermore, the Gibbs measure is not absolutely continuous with respect to the usual Wiener loop, and does not live on $L^2$ itself.

Lebowitz, Rose and Speer \cite{LRS1988} introduced invariant Gibbs measures for the nonlinear cubic Schr\"odinger equation, and proved that they can be normalised on bounded subsets of $L^2$ of the form
 \begin{equation}\label{Omega_N}
   \Omega_N=\left\{ f\in L^2({\mathbb T}; {\mathbb R}) \st \int_{\mathbb T} f(x)^2 \,\frac{dx}{2\pi} \leq N\right\}.
  \end{equation}
These measures determine the modified canonical ensemble. The fundamental idea is that solutions drawn from the support of the Gibbs measure should exhibit typical behaviour of solutions, which may not be exhibited by smooth or specially chosen solutions.

  Bourgain \cite{Bour1994} introduced Gibbs measures for the periodic KdV equation via random Fourier series, and the current paper follows this method. Lebowitz, Rose and Speer \cite{LRS1988} identified two different regimes.
\begin{enumerate}[(i)]
 \item For sufficiently small $N$, the Gibbs measure is absolutely continuous with respect to the free measure, and there is a well-posed dynamical system when the initial data lies in the support of the Gibbs measure, and the measure is invariant under the flow associated with the dynamical system.
 \item For sufficiently large $N$, the Gibbs measure tends to concentrate on a travelling wave solution, which is given by a soliton.
\end{enumerate}
Likewise, for the Benjamin--Ono equation, there are two regimes for periodic solutions.
\begin{enumerate}[(i)]
 \item  In Section~\ref{S:Gibbs}, we work in Fourier space and obtain the Benjamin--Ono equation from a Hamiltonian system with canonical
 coordinates given by the Fourier coefficients.
 Hence the properties of  the Gibbs measure are accessible by the techniques of random Fourier series. We prove that the Gibbs measure satisfies a logarithmic Sobolev inequality, and discuss further consequences of this such as transportation cost inequalities.
 \item  In Section~\ref{S:Travelling}, we introduce multi-soliton periodic travelling wave solutions, which are governed by a Hamiltonian in canonical coordinates in position and velocity space, rather than Fourier space. A further difference is that the multi-soliton is specified by a probability measure on the circle. In Sections~\ref{S:RandomM} and \ref{S:ContHS}, we analyse the $n$-solitons as $n\rightarrow\infty$, and interpret their limiting behaviour.
\end{enumerate}

\section{Definitions and notation}\label{S:Defn}

Throughout we shall identity a function $f \in L^2(\mathbb{T},\mathbb{R})$ with its Fourier coefficients,
  \[ f(x)=\frac{a_0}{2}+\sum_{j=1}^\infty (a_j\cos jx +b_j\sin jx) \]
giving an isometric isomorphism between $L^2({\mathbb T}, {\mathbb R})$ and ${\mathbb R}\oplus \ell^2 ({\mathbb N}; {\mathbb R}^{2\times 1})$.

With this convention we have the identifications
    \begin{equation}
    {\frac{\partial}{\partial x}} \leftrightarrow 0\oplus \bigoplus_{j=1}^\infty
       \begin{bmatrix}
           0&j\\
          -j&0
       \end{bmatrix},
    \qquad
    {\mathcal H}  \leftrightarrow 0\oplus \bigoplus_{j=1}^\infty
         \begin{bmatrix}
             0&1\\
            -1&0
         \end{bmatrix},
    \end{equation}
where these matrices are skew-symmetric and commute. We have the Poisson bracket,
 \[  \{ f,g\}=\sum_j j \left(-\frac{\partial f}{\partial b_j}\frac{\partial g}{\partial a_j} +
          \frac{\partial f}{\partial a_j} \frac{\partial g}{\partial b_j} \right)\]
for infinitely differentiable functions $f,g:{\ell}^{2}\rightarrow {\mathbb R}$ that depend on only finitely many coordinates. Hence the canonical equations of motion are
 \begin{equation}\label{canonical}\dot a_n =\{ a_n, H\}=n \frac{\partial H}{\partial b_n}, \qquad \dot b_n =\{ b_n, H\} =-n\frac{\partial H}{\partial a_n}. \end{equation}
In particular, with
  \begin{equation}
      H=\sum_{n=1}^M n(a_n^2+b_n^2)- \frac{\beta}{3}
          \int_{\mathbb T} \left( \sum_{n=1}^M a_n\cos n x  +b_n\sin n x  \right)^3\,  \frac{d x }{2\pi},
  \end{equation}
we obtain a finite-dimensional version of the Benjamin--Ono equation.
This is consistent with the Poisson bracket used for periodic KdV and similar evolution equations \cite{Bour1994}.

In the following definition and later, $\sumd$ denotes a sum where the term for index 0 is omitted.

\begin{defn}
For $\eta \in {\mathbb R}$, let
  \begin{equation}\label{H-eta}
   \dot H ^\eta = \left\{ f(x)=\sumd_{n=-\infty}^\infty c_n e^{inx} \st \text{$c_n\in {\mathbb C}$ and $\ds \sumd_{n=-\infty}^\infty \vert n\vert^{2\eta}\vert c_n\vert^2<\infty$}  \right\}
   \end{equation}
with $\Vert f\Vert_{H^\eta}= \Bigl( \sumd_{n} \vert n\vert^{2\eta }\vert c_n\vert^2 \Bigr)^{1/2}$, and let $H^\eta ={\mathbb C}\oplus \dot H^\eta$.
\end{defn}

Note that the canonical inclusion map $H^{1/2}\rightarrow H^{-1/2}$ is Hilbert--Schmidt, and $H^{-1/2}$ is the dual space of $H^{1/2}$ for the pairing $\langle f,g\rangle = \ds \int_{\mathbb T} f(x)\bar g(x)\, dx/(2\pi )$.

We can now define the Gibbs measures which will be analysed in Section~\ref{S:Gibbs}.

\begin{defn} (Gibbs measure)
Suppose that $\beta \in \mathbb{R}$, $N > 0$ and $M \in \mathbb{N}$.
Identifying a function $f \in L^2(\mathbb{T},\mathbb{R})$ with its Fourier coefficients as above, the probability measure $\nu_{\beta, N,M}$ on $L^2(\mathbb{T},\mathbb{R})$, with parameters $\beta$, $N$ and $M$, is defined by
\begin{multline}\label{inv-meas}
  \nu_{\beta, N,M} (df)  \\
  = \frac{1}{Z_{\beta, N,M}} {\bf I}_{\Omega_N}(f)
       \exp \left( -\beta \int_{\mathbb T} \left( \sum_{j=1}^M a_j \cos j x+b_j\sin j x\right)^3 \frac{dx}{2\pi }\right)
            \prod_{j=1}^M e^{-j(a_j^2+b_j^2)/2} \frac{da_jdb_j}{2\pi j},
  \end{multline}
where ${\bf I}_{\Omega_N}$ denotes the indicator function of $\Omega_N$, and $Z_{\beta, N,M}>0$ is the appropriate normalising constant.  These measures are called Gibbs measure for the modified canonical ensemble for the periodic Benjamin--Ono equation.
\end{defn}

We shall denote by ${\mathcal P}(\Omega)$ the set of Radon probability measures on a complete separable metric space $(\Omega,d)$, equipped with the weak topology.

In Sections \ref{S:Travelling} and \ref{S:RandomM} we discuss another Hamiltonian system associated with the soliton solutions of the Benjamin--Ono equation,
\begin{equation}\label{KHam} K_n(p,q) ={\frac{1}{2}}\sum_{j=1}^n p_j^2 +{\frac{k^2}{2}}\sum_{m,\ell=1: m\neq \ell}^n {\hbox{cosec}}^2 \left({\frac{k(q_m-q_\ell)}{2}}\right).\end{equation}
 With $\phi (x)=-ik\cot (kx/2)$, a particular collection of solutions of the canonical equations of motion is given by
   \begin{equation}\label{q'}
   {\frac{dq_\ell }{dt}}
   = -\phi (q_\ell -\bar q_\ell )-\sum_{m:m\neq \ell} \phi (q_m-q_\ell )
           -\sum_{m:m\neq \ell} \phi (q_\ell -\bar q_m)\qquad (\ell =1, \dots, n),
   \end{equation}
and the initial condition $(q_j(0))_{j=1}^n$ gives a discrete  $\omega_n=(1/n)\sum_{j=1}^n \delta_{e^{iq_j}}\in {\mathcal P}({\mathbb D})$. Then there exists a solution $u_n(x,t)$ of the Benjamin--Ono equation such that the $\omega_n$ determines the initial profile $u_n(x,0)$, and the pole
dynamics of $u_n$ is determined by the ODE (\ref{q'}). The poles $e^{iq_j}$ are known as vortices.

The system (\ref{q'}) can itself be described as a dynamical system with Hamiltonian $E_{n,v}$, which amounts to the electrostatic energy of a collection of points in the complex plane. This $E_{n,v}$ has a Gibbs measure on phase space, which is the space of initial conditions of the dynamics of $\omega_n$.  In Section \ref{S:RandomM} we show how  $E_{n,v}$ arises from a  random matrix model. Such models are often called log gas models due to the formula for the electrostatic energy. Using the techniques of random matrix theory, we obtain concentration of measure results for these Gibbs measures as $n\rightarrow\infty$.

In Section \ref{S:ContHS}, we suppose that $\rho_n\rightarrow \rho$ weakly in ${\mathcal P}({\mathbb T})$, where the density $\rho$ satisfies $Q'={\mathcal H}\rho $ for some $Q\in C^2({\mathbb T}; {\bf R})$, which is known as the scalar potential. As $n\rightarrow\infty$, the sequence of Hamiltonians $\{K_n\}$ converges to the Hamiltonian $K$, where
 \[ K(q, \rho )
     = {\frac{1}{2}}\int_{\mathbb T} \rho (x)\left( {\frac{\partial q}{\partial x}}\right)^2\, dx
        +{\frac{2\pi^2}{3}}\int_{\mathbb T} \rho (x)^3\, dx. \]
This $K$ gives an autonomous Hamiltonian system on an infinite-dimensional phase space $L^2({\mathbb T})\times L^2({\mathbb T})$  and canonical variables $(\rho, q)$, for functions $q, \rho: {\mathbb T}\rightarrow {\mathbb R}$.  This gives rise to the system of partial differential equations
   \begin{align}
     &{\frac{\partial \rho}{\partial t}}+{\frac{\partial}{\partial x}}\left( \rho {\frac{\partial q}{\partial x}}\right)
      = 0, \label{conservation} \\
     &{\frac{\partial q}{\partial t}}+{\frac{1}{2}}\left( {\frac{\partial q}{\partial x}}\right)^2+2\pi^2\rho^2
       =0, \label{limitpole}
   \end{align}
which are a version of the isentropic Euler equations for a compressible gas. Here $\rho$ (assumed to lie in $L^3(\mathbb{T})\cap {\mathcal P}(\mathbb{T})$) represents the gas density, and
 equation (\ref{conservation}) describes the conservation of mass; let $L^2(\rho )=\{ g:{\mathbb T}\rightarrow {\mathbb C}: \int \vert g (e^{i\theta})\vert^2\rho (e^{i\theta})\, d\theta<\infty\}$. With $q$ regarded as a phase variable,
 $u={\frac{\partial q}{\partial x}}$ is the velocity with $u\in L^2(\rho)$. The equation (\ref{limitpole}) is a modified version of the Hamilton--Jacobi equation in which we regard $q$ as a phase function. One can therefore regard ${\mathcal P}_2({\mathbb T})$ as a manifold with a tangent space at $\rho_0$ which is modelled on $H^{-1/2}$. See   \cite[Theorem~7.26]{V2003} for a related result. The significance of the pairing $(H^{1/2}, H^{-1/2})$ is its universality, depending only on the general properties of $\rho_0$.

\section{Concentration of the Gibbs measure}\label{S:Gibbs}

In this section, we prove a logarithmic Sobolev inequality for the Gibbs measure $\nu_{\beta, N,M}$. We suppress the $t$ variable, and consider the Hamiltonian $H_\beta (u)$ defined for $u=u(x)$.  The operators ${\mathcal H}$ and $\vert D\vert$ are not local in the sense of \cite{BakEm}, so $\vert D\vert$ does not give rise to a \textsl{carr\'e du champ} operator in the style of Bakry and {\'E}mery. We circumvent this problem by using the Fourier coefficients as the primary variables. The following lemma enables us to interpret the spatially periodic Benjamin--Ono equation in terms of Fourier series for $u$ in the $x$ variable and $u$ in the domain of $\vert D\vert^{1/2}$, and we work on $\Omega_N\cap H^{1/2}$ with coordinates in Fourier transform space. We identify $H^{1/2}$ with a Dirichlet space of harmonic functions.
For $0<r<1$, let
  \[ P_r(\theta )= \frac{1-r^2}{1-2r\cos\theta +r^2}
     = \sum_{n=-\infty}^\infty r^{\vert n\vert} e^{in\theta}
  \]
be the usual Poisson kernel.

\begin{lem}
\begin{enumerate}[(i)]
\item
The Poisson semigroup  $(P_{e^{-t}})_{t\geq 0}$ on $L^2({\mathbb T})$ has generator $\vert D\vert^{1/2}$ with domain $H^{1/2}$.
\item Every $u\in H^{1/2}$ may be identified with a real harmonic function $U:{\mathbb D}\rightarrow{\mathbb R}$ such that $U(re^{i\theta})=P_ru(\theta )$ and 
    \begin{equation}
    \Vert u\Vert_{H^{1/2}}^2
       = \vert U(0)\vert^2
            +{\frac{1}{\pi}}\int\!\!\!\int_{\mathbb D} \Vert\nabla U (re^{i\theta })\Vert^2\,r \,dr d\theta.
    \end{equation}
\item For all $u\in H^{1/2}$, the function $e^{u}$ is integrable with
   \begin{equation}
   \log \int _{\mathbb T} e^{u (\theta)}{\frac{d\theta}{2\pi}}
      \leq {\frac{1}{4\pi}}\int\!\!\!\int_{\mathbb D} \Vert \nabla U (re^{i\theta })\Vert^2\,r \, dr d\theta
                 +\int _{\mathbb T}u (\theta) \, {\frac{d\theta}{2\pi}}.
   \end{equation}
\item Let $\varphi : {\mathbb D}\rightarrow {\mathbb D}$ be a one-to-one analytic function such that $\varphi (0)=0$. Let $ u\circ \varphi\in H^{1/2}$ denote the boundary values of the harmonic function $U(\varphi (re^{i\theta}))$. Then the Hamiltonian $H_\beta$ has the property that  $u\mapsto H_\beta(u\circ\varphi )$ is continuous on $\Omega_N\cap H^{1/2}$ for the metric of $H^{1/2}$.
\end{enumerate}
\end{lem}

\begin{proof} (i) Every $u\in L^2({\mathbb T}, {\mathbb R})$ can be extended to a real harmonic function $U(re^{i\theta })=P_ru(\theta)=\int_{\mathbb T} u(\phi )P_r(\theta -\phi )d\phi /(2\pi )$ with $U:{\mathbb D}\rightarrow {\mathbb R}$ and $P_r:e^{in\theta}\mapsto r^{\vert n\vert} e^{in\theta}$.

(ii) This follows from (i).

(iii) This is the Milin--Lebedev inequality, as in ($4'$) from \cite{OPS}.

(iv) By (iii), we have $u\in L^3$ for all $u\in H^{1/2}$, so $H_\beta (u)$ is well defined. Then $U\circ\varphi :{\mathbb D}\rightarrow {\mathbb R}$ is also harmonic and has boundary values $u\circ\varphi$, where
  \[ \Vert u\circ\varphi \Vert_{H^{1/2}}^2
      =\vert U(0)\vert^2
           +{\frac{1}{\pi}}\int\!\!\!\int_{\varphi ({\mathbb D})} \Vert\nabla U (re^{i\theta })\Vert^2\, r\, dr d\theta
   \]
and $\varphi ({\mathbb D})\subseteq {\mathbb D}$. By Littlewood's subordination principle, for $1<p<\infty $ there exists $C_p>0$ such that
   \[ \Vert  u\circ \varphi \Vert_{L^p}\leq C_p\Vert u\Vert_{L^p}\qquad (u\in L^p({\mathbb T}; {\mathbb R})).\]
We can therefore write
    \begin{align*}
  H_\beta (u\circ\varphi )-H_\beta (v\circ\varphi )
      &=\Vert u\circ \varphi\Vert^2_{H^{1/2}}-\Vert v\circ\varphi\Vert^2_{H^{1/2}} \\
      &\qquad +\beta \int_{{\mathbb T}} \bigl( (u\circ\varphi )^3-(v\circ \varphi )^3\bigr) {\frac{d\theta}{2\pi}} \\
      &\leq \Vert u\circ\varphi-v\circ\varphi\Vert_{H^{1/2}}
             \bigl(\Vert u\circ\varphi\Vert_{H^{1/2}} +\Vert v\circ\varphi\Vert_{H^{1/2}}\bigr)   \\
      &\qquad +\beta \Vert u\circ\varphi -v\circ\varphi\Vert_{L^2}
             \bigl( \Vert u\circ\varphi\Vert_{L^4}+ \Vert v\circ\varphi\Vert_{L^4}\bigr)^2,
    \end{align*}
which by subordination is bounded by
  \begin{multline*}
 \Vert u-v\Vert_{H^{1/2}}\bigl(\Vert u\Vert_{H^{1/2}}+\Vert v\Vert_{H^{1/2}}\bigr)+C_4\beta \Vert u-v\Vert_{L^2}\bigl( \Vert u\Vert_{L^4}+ \Vert v\Vert_{L^4}\bigr)^2 \\
 \leq  \Vert u-v\Vert_{H^{1/2}}\bigl(\Vert u\Vert_{H^{1/2}}+\Vert v\Vert_{H^{1/2}}\bigr)+C_4\beta \Vert u-v\Vert_{H^{1/2}}\bigl( \Vert u\Vert_{H^{1/2}}+ \Vert v\Vert_{H^{1/2}}\bigr)^2.
 \end{multline*}
\end{proof}

There are natural inclusion maps $\dot H^{1/2}\rightarrow L^2\rightarrow H^{-1/2}$, so $H^{-1/2}$ is the dual space of $\dot H^{1/2}$. Let $\Log :H^{-1/2}\rightarrow H^{1/2}$ be the map
   \begin{equation*}
   \Log f(x)=\int_{\mathbb T} \log {\frac{1}{4\sin^2 ((x-y)/2)}}\, f(y)\, {\frac{dy}{2\pi}} + {\hat f}(0),
   \end{equation*}
which may be expressed in terms of convolution with respect to
$\sum'{\frac{e^{i n \theta}}{\vert n\vert}}=-\log (4\sin^2(\theta /2))$.
Note  that there is a bounded bilinear multiplication
  \[  H^{1/2}\times H^{1/2}\rightarrow H^{1/2} :\quad (u,v)\mapsto \Log (uv),  \]
so there exists $M>0$ such that
  $\Vert \Log (uv)\Vert_{H^{1/2}} \leq M\Vert u\Vert_{H^{1/2}}\Vert v\Vert_{H^{1/2}}$
for all $u,v\in H^{1/2}$
as one checks by estimating the Fourier coefficients. We deduce that for $u\in H^{1/2}$, the term
  ${\mathcal H}{\frac{\partial u}{\partial x}}+\beta u^2$
lies in $H^{-1/2}$ and hence is a distribution in the dual of $H^{1/2}$.

\begin{prop}\label{prop2.1}
There exists $\kappa >0$ such that if $\vert \beta\vert \sqrt{N}<1/\kappa$ then the Hamiltonian $H_\beta$ defined in (\ref{Hamilt}) is uniformly convex on $\Omega_N$. Furthermore $\nu_{\beta, N,M}$ satisfies the logarithmic Sobolev inequality,
 \begin{equation}\label{log-sob1}
 \int_{\Omega_N} F(u) \log \Bigl( F(u)^2 /\int F^2 \, d\nu_{\beta, N,M} \Bigr)\, \nu_{\beta ,N,M} (du)
   \leq \frac{2}{\alpha} \int_{\Omega_N} \Vert \nabla F(u) \Vert^2_{L^2} \, \nu_{\beta ,N,M} (du)
 \end{equation}
where $\alpha =1-\kappa \vert \beta\vert \sqrt{N}$ and $\nabla F$ is the Fr{\'e}chet derivative of $F\in C^1 (\Omega_N; {\mathbb R})$.
\end{prop}

\begin{proof}
With
  \[ \nabla f =\left( \frac{\partial f}{\partial a_j}, \frac{\partial f}{\partial b_j}\right)_{j=1}^\infty \]
we have the \textit{carr{\'e} du champ it{\'e}r{\'e}} operator from \cite{BakEm},
   \begin{equation}
      \Gamma_2(\psi )
        =\Vert {\hbox{Hess}}( \psi )\Vert^2_{HS}+ {\hbox{Hess}}(H_\beta )(\nabla \psi \otimes  \nabla\psi )
            +{\hbox{Ric}}_{\Omega_N}(\nabla \psi \otimes  \nabla\psi ).
   \end{equation}
To satisfy the Bakry--{\'E}mery condition \cite{BakEm}, it suffices to produce $\alpha>0$ such that
  \[ \Gamma_2 (\psi )\geq \alpha \Vert \nabla\psi \Vert^2_{L^2}
  \]
for all smooth $\psi :\Omega_N\rightarrow {\mathbb R}$. The terms 
$\Vert {\hbox{Hess}}( \psi )\Vert^2_{HS}$ and ${\hbox{Ric}}_{\Omega_N}(\nabla \psi \otimes  \nabla\psi )$ are nonnegative, so it suffices to show ${\hbox{Hess}}(H_\beta )(\nabla \psi \otimes \nabla \psi)\geq \alpha\Vert \nabla\psi \Vert^2_{L^2}$.  
Let
  \[ J_{1/2} (x )=\sum_{n=1}^\infty \frac{\cos n x}{\sqrt {n}}=\sum' {\frac{e^{inx}}{2\sqrt{\vert n\vert}}}\]
which is even and $2\pi$-periodic.
Standard estimates \cite[II:13--11]{Zyg} show that $J_{1/2}(x)- \left(\frac{\pi}{2x}\right)^{1/2}$ is bounded on $(0, \pi ]$ and hence that $J_{1/2} \in L^p(\mathbb{T})$ for $p < 2$. By Young's convolution inequality
  \[ \Vert J_{1/2}\ast h \Vert_{L^3}
     \leq \Vert J_{1/2}\ast h\Vert_{L^4}
     \leq \Vert J_{1/2}\Vert_{L^{4/3}} \, \Vert h \Vert_{L^2},\]
so $h \mapsto H_\beta (J_{1/2}\ast h)$ is bounded on $\Omega_N\subset L^2$. Thus we regard $H_\beta$ as a densely defined function on $\Omega_N$. The
Hessian is defined as a quadratic form on $H^{1/2}$ by
 \begin{equation}  
 \bigl\langle {\hbox{Hess}} (H_\beta)(u) , h\otimes h\bigr\rangle
    =\int_{\mathbb T} \HT \frac{dh}{dx}\, h(x) \,\frac{dx}{2\pi}
           +\beta \int_{\mathbb T} u(x) h(x)^2 \,\frac{dx}{2\pi}.
 \end{equation}
Observe that
 \begin{equation}\label{Fix}
 \bigl\langle {\hbox{Hess}} (H_\beta)(u) , J_{1/2}\ast h\otimes J_{1/2}\ast h\bigr\rangle
     =\int_{\mathbb T} h(x)^2 \, \frac{dx}{2\pi}
         +\beta \int_{\mathbb T} u(x)(J_{1/2}\ast h)(x)^2\, \frac{dx}{2\pi},
  \end{equation}
where
  \begin{align*}
  \beta \int_{\mathbb T} u(x)(J_{1/2}\ast h)(x)^2 \frac{dx}{2\pi}
      &\geq -\vert \beta\vert \ \Vert u\Vert_{L^2} \  \Vert J_{1/2}\ast h\Vert^2_{L^4}   \\
      &\geq -\vert\beta\vert \ \sqrt{N}\ \Vert J_{1/2}\Vert^2_{L^{4/3}}\ \Vert h\Vert^2_{L^2}.
  \end{align*}
Hence with $\kappa =\Vert J_{1/2}\Vert^2_{L^{4/3}}$ and $ h=\vert D\vert^{1/2}g$, we have
  \[ \bigl\langle {\hbox{Hess}} (H_\beta)(u) , g\otimes g\bigr\rangle
      \geq \bigl( 1- \kappa \vert \beta\vert N^{1/2}) \Vert g\Vert_{L^2}^2.
  \]
Thus $H_\beta$ is uniformly convex on $\Omega_N$.

It now follows by the Bakry--{\'E}mery criterion that the measures $\nu_{\beta, N,M}$ satisfy the logarithmic Sobolev inequality with constant
$\alpha =1-\kappa \vert \beta\vert N^{1/2}$.
\end{proof}

For more general $\beta$ and $N$ we can use the Holley--Stroock Theorem to deduce a logarithmic Sobolev inequality, although with a possibly much poorer constant.

\begin{thm}\label{logsobthrm}
For all $\beta, N>0$ the Gibbs measure $\nu_{\beta, N,M}$ of the periodic Benjamin--Ono equation satisfies the logarithmic Sobolev inequality (\ref{log-sob1}) with constant
  \begin{equation} \alpha \geq \frac{1}{2} \exp (-14\beta N^{3/2}(cN\beta^{1/2}+1)-N/2)\end{equation}
  for come absolute constant $c>0$.
\end{thm}

\begin{proof} Rather than work directly with the Hamiltonian $H_\beta$ we first show a Bakry--{\'E}mery condition for a perturbed Hamiltonian of the form $H_\beta - W$.

We choose some integer $k$ such that $c^4\beta^4 N^2 +1\ge k \ge c^4\beta^4 N^2$, where $c>0$ is a large constant to be chosen below. One can split any function $u( x  )=\sum_{n=-\infty}^\infty \hat u(n) e^{in x }$ in $L^2(\mathbb{T})$ into a sum $u_H( x  )=\sum_{n=-k}^k \hat u(n) e^{in x }$ consisting of the low Fourier modes, and a tail $u_T( x  )=\sum_{\vert n\vert >k} \hat u(n) e^{in x }$ consisting of the high Fourier modes. This is an orthogonal decomposition so if $u\in \Omega_N$, then $u_H, u_T\in \Omega_N$ too. Furthermore, if $u \in \Omega_N$, the Cauchy--Schwarz inequality applied to the Fourier coefficients of $u_H$ gives $\Vert u_H\Vert_\infty\leq \sqrt{2k+1} N^{1/2}$.

Define the potential $W$ on $\Omega_N$ by
    \[ W(u)=-\beta\int_{{\mathbb T}} u_H( x  )^3 \frac{d x }{2\pi}
       -3\beta \int_{{\mathbb T}} u_H( x  )^2u_T( x ) \frac{d x }{2\pi}
         -3\beta\int_{{\mathbb T}} u_H( x  )u_T( x  )^2 \frac{d x }{2\pi}
         -\frac{1}{2}\vert \hat u(0)\vert^2.
  \]
So we have 
   \begin{align}
   \vert W(u)\vert&\leq \vert\beta \Vert u_H\Vert_{L^2}^2\Vert u_H\Vert_{L^\infty}+3\vert\beta\vert \Vert u_H\Vert_{L^\infty}\Vert u_H\Vert_{L^2}\Vert u_T\Vert_{L^2}+3\beta \Vert u_H\Vert_{L^\infty}\Vert u_T\Vert^2_{L^2}+{{1}\over{2}}\Vert u\Vert_{L^2}^2\nonumber\\
   &\leq 7\beta \sqrt{2k+1} N^{3/2}+\frac{N}{2}\nonumber\\
   &\leq 14\beta N^{3/2}(c\beta^{1/2}N+1)+\frac{N}{2},
   \end{align}
and so $W$ is bounded.

Consider now
  \[ H_\beta (u)-W(u)
  =\frac{\vert \hat u(0)\vert^2}{2}
        +\frac{1}{2}\sum_{n=-\infty}^\infty \vert n\vert \vert\hat u(n)\vert^2
        -\beta \int_{\mathbb T} u_T( x  )^3\, d x .
  \]
On computing the Hessian, one finds
  \begin{align} {\hbox{Hess}}(H_\beta -W)(u) (f\otimes f)&=\lim_{\varepsilon\rightarrow 0+}\varepsilon^{-2}\bigl( (H_\beta -W)(u+\varepsilon f) +(H_\beta -W)(u-\varepsilon f) \nonumber \\
  & \qquad -2(H_\beta -W)(u)\bigr)\nonumber\\
    &=\vert \hat f(0)\vert^2
        +\sum_{n=-\infty}^\infty \vert n\vert \vert \hat f(n)\vert^2
        -3\beta \int_{\mathbb T} u_T( x  )f_T( x  )^2\, \frac{d x }{2\pi}.
  \end{align}
Now
   \begin{align*}
   \left\vert3\beta\int u_T(x )f_T(x)^2dx\right\vert
     &\leq 3 \beta  \Vert u_T\Vert_{L^2} \, \Vert \Delta^{-3/16}\Delta^{3/16}f_T\Vert^2_{L^4}\\
     &\leq 3 \beta \Vert u_T \Vert_{L^2}
            \, \left\Vert  \sumd {\frac{e^{-inx}}{\vert n\vert^{3/8}}}\right\Vert^2_{L^{4/3}}
            \, \bigl\Vert\Delta^{3/16}f_T\bigr\Vert^2_{L^2}.
   \end{align*}
Standard estimates \cite[Section V.2]{Zyg} show that
\[ \sumd {\frac{e^{-inx}}{\vert n\vert^{3/8}}} = \frac{2\, \Gamma(\frac{5}{8})}{x^{5/8}} + O(1), \]
and so
  \[ \int_0^{2\pi} \left\vert \sumd {\frac{e^{-inx}}{\vert n\vert^{3/8}}} \right\vert^{4/3} {\frac{dx}{2\pi}}
     \leq C\int_0^{2\pi} {\frac{dx}{  x ^{5/6}}}  = C_1,
  \]
say. Also, the large Fourier modes satisfy
\begin{align} \bigl\Vert \Delta^{3/16}f_T \bigr\Vert_{L^2}^2&
     \leq \sum_{n=-\infty}^{-k} \vert n\vert^{3/4} \vert \hat f(n)\vert^2+\sum_{n=k}^\infty  \vert n\vert^{3/4} \vert \hat f(n)\vert^2\nonumber\\
     &\leq   k^{-1/4}\sum_{n=-\infty}^{-k}\vert n\vert \vert \hat f(n)\vert^2+\sum_{n=k}^\infty  \vert n\vert \vert \hat f(n)\vert^2\nonumber\\
     &\leq k^{-1/4} \Vert f \Vert^2_{H^{1/2}}. \end{align}
Thus
  \[ \left\vert3\beta\int u_T(x )f_T(x)^2dx\right\vert
      \leq 3 \beta \sqrt{N} C_1 k^{-1/4} \Vert f \Vert^2_{H^{1/2}}
      \leq \frac{3C_1}{c} \Vert f\Vert^2_{H^{1/2}},   \]
and so
   \[ {\hbox{Hess}}(H_\beta -W)(u) (f\otimes f)
   \geq \left(1-\frac{3C}{c}\right) \Vert f\Vert^2_{H^{1/2}},
   \]
so that $H_\beta -W$ satisfies the Bakry--{\'E}mery condition for suitably large $c$.

By the Holley--Stroock theorem \cite[p.~1184]{HS}, $H_\beta =(H_\beta-W)+W$ is a bounded perturbation of a uniformly convex potential, hence satisfies a logarithmic Sobolev inequality with constant
  \[ \alpha \ge  2^{-1}\exp (-\Vert W\Vert_{\infty}) \geq 2^{-1}\exp
   \bigl(-14\beta N^{3/2}(cN\beta^{1/2} +1)-N/2\bigr). \]
\end{proof}

\section{Transport inequalities}\label{S:Transport}

Transport inequalities relate the cost of transporting a probability measure $\omega$ onto a particular reference measure $\rho_0$ with some functional such as the relative entropy of $\omega$ with respect to $\rho_0$.
There are well-known connections between such inequalities and concentration of measure inequalities and logarithmic Sobolev inequalities. In this section we use the results of Section~\ref{S:Gibbs} to obtain a result of this type.

Given a periodic $C^2$ real potential function $Q(e^{i\theta})$, we consider the energy functional
   \begin{equation*}
     E_Q(\omega )=\int_{\mathbb T} Q(e^{it})\, \omega (dt)+\int\!\!\!\int_{{\mathbb T}^2\setminus\Delta}\log {\frac{1}{\vert e^{i\psi}-e^{it}\vert}}\, \omega (d\psi )\omega (dt),
     \end{equation*}
where $\omega\in {\mathcal P}({\mathbb T})$ and $\Delta =\{ (e^{i\psi}, e^{it}): t=\psi\}$. Then we define the minimum energy by
   \begin{equation*}
   E_Q=\inf_\omega \bigl\{E_Q(\omega ):\omega \in {\mathcal P}({\mathbb T})\bigr\}.
   \end{equation*}
We shall denote by $\rho_0$ the probability density function of the unique measure that minimises $E_Q$, and satisfies
   \begin{equation}\label{defQ}Q(e^{i\theta})
      =2\int_{\mathbb T} \log \vert e^{i\theta} -e^{i\phi}\vert \, \rho_0 (\phi)\, d\phi+C
   \end{equation}
for some constant $C$. We call $Q$ the potential corresponding to the equilibrium density $\rho_0$.

\begin{defn} Suppose that $\omega \in {\mathcal P}({\mathbb T})$.
\begin{enumerate}[(i)]
\item The relative free entropy of $\omega$ with respect to $\rho_0$ is
   \begin{equation}\label{deffreeent}
     \tilde \Sigma_Q(\omega )=\Sigma (\omega \mid \rho_0)=E_Q(\omega )-E_Q.
   \end{equation}
\item  If further $\omega, \rho_0\in L^3$, then the relative free information of $\omega$ with respect to $\rho_0$ is
   \begin{equation}\label{deffreeinf}
   I_F(\omega\mid \rho_0)
       =\int_{\mathbb T}\bigl( {\mathcal H}(\omega -\rho_0)(x)\bigr)^2 \, \omega ( dx).
   \end{equation}
\end{enumerate}
\end{defn}

Clearly $\tilde \Sigma_Q(\omega )\geq 0$, and $\tilde \Sigma_Q(\omega )=0$ if and only if $\omega (d\theta )=\rho_0(\theta )\,d\theta $ up to a set of zero capacity.

The density $\rho_0$ satisfies the principal value integral equation
   \[
   {\frac{d}{d\theta}}Q(e^{i\theta})
      =2\, {\hbox{p.v.}}\int_{\mathbb T} \cot {\frac{\theta-\psi}{2} }\, \rho_0(\psi )d\psi .
   \]
We are especially interested in the case in which there exists $\kappa>0$ such that
   \begin{equation}\label{freecurv}
     {\frac {d^2}{d\theta^2}}Q(e^{i\theta})\geq \kappa -1/2\qquad (e^{i\theta} \in {\mathbb T}).
   \end{equation}
Whereas the minimiser $\rho_0(\theta )\, d\theta$ is absolutely continuous, the properties of $\rho_0$ are obtained via an approximation in which  $\rho_0(\theta )\, d\theta$ is the weak limit of a convex combination of $n$ point masses as $n\rightarrow\infty$; \cite{HPU}. Using $Q$, one can introduce a probability measure (\ref{matrixdist}) on the group $U(n)$ of $n\times n$ unitary complex matrices which is invariant under unitary conjugation $X \mapsto UXU^\dagger$. The
typical $X \in U(n)$ has eigenvalues $e^{i\theta_1}, \dots ,e^{i\theta_n}$, and empirical eigenvalue distribution
$\omega_n =n^{-1}\sum_{j=1}^n \delta_{e^{i\theta_j}}$, and as $n\rightarrow\infty$, the $\omega_n$ converge weakly almost surely to $\rho_0 (x)\, dx.$  

The statistical properties of $\omega_n$ for $n$ large are described by eigenvalues of matrices $X \in U(n)$, where the probability measure on $U(n)$ is given by the Haar probability measure and a scalar potential $v(x)=2\int_{\mathbb T} \log \vert e^{ix}-e^{iy}\vert \rho_0(y)\, dy.$ There are quantitative results describing the convergence of $\omega_n\rightarrow \rho_0$ in the weak topology.

\begin{defn} (Wasserstein distance) Let $({\mathcal X},d)$ be a compact metric space.  The  Wasserstein metric on ${\mathcal P}({\mathcal X})$ is
   \[
      W_2(\mu_0, \mu_1)
         =\inf_\pi \left\{ \left( \int\!\!\!\int_{{\mathcal X}^2}d(x,y)^2\, \pi (dxdy)\right)^{1/2}\right\},
   \]
where the infimum is taken over all the probability measures $\pi$ on ${\mathcal X}^2$ that have marginals $\mu_0(dx)$ and $\mu_1(dy)$. We shall use the notation ${\mathcal P}_2({\mathcal X})$ as shorthand for the metric space $(\mathcal{P}({\mathcal X}),W_2)$.
\end{defn}

We introduce the following temporary definition, which we later reconcile with more standard definitions, as in (\ref{BQformula}) and (\ref{freetrans}).

\begin{defn}
We shall say that a probability density function $\rho_0$ on ${\mathbb T}$ satisfies the free transportation inequality if there exists $C>0$ such that
    \begin{equation}\label{FTI}
        W_2(\rho, \rho_0)\leq C\Vert\rho -\rho_0\Vert_{H^{-1/2}}
   \end{equation}
for all probability density functions $\rho$ on ${\mathbb T}$.
\end{defn}

\begin{prop}\label{prop4.4}
Let $\nu$ be a probability measure on $(\Omega_N, L^2)$ that is absolutely continuous with respect to $\nu_{\beta, N,M}$. Then, under the hypotheses of Theorem \ref{logsobthrm}, $\nu_{\beta, N}$ satisfies the transportation cost inequality
   \begin{equation}\label{Tal}
      W_2^2(\nu, \nu_{\beta, N,M})
          \leq {\frac{2}{\alpha}} \int_{\Omega_N} \log \left({\frac {d\nu}{d\nu_{\beta ,N,M}}}\right)\, d\nu.
   \end{equation}
\end{prop}

\begin{proof}
This follows from Theorem  \ref{logsobthrm} by a result of Otto and Villani; see \cite[p.~292]{V2003}.
\end{proof}

The right-hand side of (\ref{Tal}) involves the relative entropy of $\nu$ with respect to $\nu_{\beta, N,M}$. Next we consider a version of the free energy. For $f\in H^{-1/2}$ and $u\in H^{1/2}$ write $\langle f,u\rangle =\ds \int_{\mathbb T} f(x)u(x) \, \frac{dx}{2\pi}$. Let
   \[\phi (f)
     =\log\int_{\Omega_N} \exp (\langle f,u\rangle ) \, \nu_{\beta, N} (du)
         -\int_{\Omega_N}  \langle f,u\rangle \, \nu_{\beta, N} (du) \qquad (f\in H^{-1/2})
    \]
be the logarithmic moment generating function.

\begin{cor}\label{cor2.5} The logarithmic moment generating function 
$\phi $ is a convex function on $H^{1/2}$ and there exists $C(\beta , N)$ such that $\vert \phi (f)\vert\leq C(\beta, N)\Vert f\Vert^2_{H^{-1/2}}$  for all $f\in H^{-1/2}$.
\end{cor}

\begin{proof} By H\"older's inequality, $f\mapsto \phi (f)$ is a convex function. Also, $u\mapsto \langle f,u\rangle$ is Lipschitz on $\Omega_N$ with respect to the norm of $H^{1/2}$, with Lipschitz constant $\Vert f\Vert_{H^{-1/2}}$,  and mean
    \[\mu =\int_{\Omega_N}  \langle f,u\rangle\, \nu_{\beta, N} (du). \]
The logarithmic Sobolev inequality implies that $\nu_{\beta ,N}$ satisfies a Gaussian-style concentration inequality by \cite[Theorem 22.17]{V2009}, which in this case gives
   \[ \int_{\Omega_N} \exp \bigl(\langle f,u\rangle \bigr)\, \nu_{\beta, N} (du)
    \leq \exp \bigl( C(\beta, N)\Vert f\Vert^2_{H^{-1/2}}\bigr),\]
for some constant $C(\beta ,N)>0$ independent of $f$.
 We deduce that $\phi$ satisfies the concentration inequality.
\end{proof}

\section{Travelling wave solutions}\label{S:Travelling}

We now contrast the solutions described by convexity result, Proposition \ref{prop2.1} with the travelling wave solution, which we obtain from the Poisson kernel for ${\mathbb D}$.  Let $c=(1+r^2)/(1-r^2)$, and let
  \begin{equation}\label{w-def}
    w(x,t) =\frac{-1}{\beta } P_r(x-ct)
  \end{equation}
which has $N(w)=\beta^{-2}(1-r^2)^{-1}(1+r^2)$, so that $N(w)\rightarrow\infty$ as $r\rightarrow 1-$.

\begin{prop}\label{prop3.1}
There exists $r<1$ such that the function
$w$ defined in (\ref{w-def}) gives a travelling wave solution of (\ref{BO-eqn}) such that $u\mapsto H_\beta (u)$ is not convex at $u=w$.
\end{prop}

\begin{proof}
 Let $f(\theta )=-\beta^{-1} P_r(\theta )$. Using the expansion $P_r(\theta )=\sum_{n=-\infty}^\infty r^{\vert n\vert} e^{in\theta}$, one can verify that
 \[ \HT f'(\theta ) = -\frac{1}{\beta} 
                   \sum_{n=1}^\infty n r^n \bigl(e^{in \theta} + e^{-in \theta} \bigr)
       = \frac{4 r^2 - 2(r^2+1)r \cos(\theta)}{\beta(1-2 r \cos(\theta)+r^2)^2}, \]
and consequently that $\HT f'(\theta ) = -c f(\theta) - \beta f(\theta)^2$. Thus, if $w(x,t)=f(x-ct)$ then
  \begin{align*}
     \frac{\partial w}{\partial t} + \HT \frac{\partial^2 w}{\partial x^2}
                           + 2 \beta \frac{\partial w}{\partial x} \, w
     &= \frac{\partial w}{\partial t} 
          - \frac{\partial\ }{\partial x}\HT \frac{\partial w}{\partial x}
                           + 2 \beta \frac{\partial w}{\partial x} \, w \\
     &=c f'(x-ct) 
          - \frac{\partial\ }{\partial x} (c f(x-ct) + \beta f(x-ct)^2) \\
     & \qquad \qquad
               + 2 \beta f'(x-ct) f(x-ct) \\
     &= \frac{1}{\beta} \bigl( c f'(x-ct) - c f'(x-ct) - 2 \beta f'(x-ct) f(x-ct)  \\
     & \qquad \qquad
         + 2 \beta f'(x-ct) f(x-ct)\bigr) \\
     &=0.
 \end{align*}
That is, 
$w$ gives a solution of (\ref{BO-eqn}) in the form of a travelling wave with speed $c$.

Repeating the calculation of (\ref{Fix}), we find that
 \begin{align}\label{Hess-bound}
 \bigl\langle {\hbox{Hess}} (H_\beta)(w) , J_{1/2}\ast h\otimes J_{1/2}\ast h\bigr\rangle
    &=\int_{\mathbb T} h(\theta )^2 \,\frac{d\theta}{2\pi}
        +\beta \int_{\mathbb T} f(\theta )(J_{1/2}\ast h)(\theta )^2 \,\frac{d\theta}{2\pi} \notag \\
    &=\int_{\mathbb T} h(\theta )^2 \,\frac{d\theta}{2\pi}
        -\int_{\mathbb T} P_r(\theta )(J_{1/2}\ast h)(\theta )^2 \,\frac{d\theta}{2\pi} \notag \\
    &\rightarrow \int_{\mathbb T} h(\theta )^2 \,\frac{d\theta}{2\pi}-(J_{1/2}\ast h)(0)^2
  \end{align}
as $r\rightarrow 1^-$. Now
  \[
  \sup\left\{ \sumd_{j=-M}^M \frac{\hat h(j)}{2\sqrt{\vert j\vert}}
       \st\sum_{j=-M}^M \vert \hat h(j)\vert^2\leq 1\right\}
  =\left(\sum_{j=1}^M \frac{1}{j}\right)^{1/2}
  \to \infty
  \]
as $M\rightarrow\infty$. We deduce that (\ref{Hess-bound}) is negative for suitably chosen $h$, so ${\hbox{Hess}}(H_\beta )(w)$ is not positive semidefinite, so $u\mapsto H_\beta (u)$ is not convex near to $w$.
\end{proof}

\begin{rem} (i) Propositions \ref{prop3.1} and \ref{prop4.4} indicate a disjunction between the typical solutions in the support of the Gibbs measure and the travelling wave solution.

(ii) Amick and Toland \cite{AmT} established that bounded and periodic solutions of Benjamin--Ono satisfy a rather stringent uniqueness theorem; our result does not contradict their Section~\ref{S:RandomM}, since $P_r$ is unbounded as $r\rightarrow 1-$. Let $u$ be a bounded real function such that
   \begin{equation}\label{ibono}
      u(x)^2-u(x)={\mathcal H}{\frac{\partial}{\partial x}} u,\qquad (x\in {\mathbb R})
   \end{equation}
and suppose that $c=u(x)^2-u(x)$ is a constant. Suppose that $u$ extends via the Poisson kernel to a harmonic function $u(x,y)$ on $\{ x+iy: y \geq 0\}$ such that
   \begin{align*}
     &\Delta u(x,y)
        =0\qquad (y>0),\nonumber\\
     &{\frac{\partial u}{\partial y}}(x,0)=u(x,0)-u(x,0)^2\qquad (x\in {\mathbb R}),\nonumber\\
     &\text{$u(x+iy)$ is bounded for $y>0$.}
   \end{align*}
Then $u$ has a harmonic conjugate $v$ such that $f=u+iv$ is analytic on $\{ x+iy: y \geq 0\}$, and $f(z)=u(z)+iv(z)$ satisfies $u(x+iy)=\Re  f(z)$, $f(0)=u(0,0)$ and
\begin{equation}\label{odef}{\frac{df}{dz}}={\frac{i}{2}}\bigl( f(z)^2+c\bigr).\end{equation}
Thus solutions to the elementary ordinary differential equation (\ref{odef}) give rise to solutions of (\ref{ibono}), and hence to travelling wave solutions of Benjamin--Ono. The same (\ref{odef}) can also be viewed as a version of the complex Burgers equation.

(iii) Comparing our results with those of Benjamin \cite{Ben1967}, we find it convenient to change variables. We introduce the strip $\Sigma_1=\{ \zeta\in {\mathbb C}: 0<\Im \zeta<1\}$ which is conformally equivalent to the upper half plane $\{ z\in {\mathbb C}: \Im z>0\}$ by
$\zeta\mapsto z=e^{\pi \zeta }$. For $\Sigma_1$ we have the Poisson kernel
  \[ P(x-\xi ;t) ={\frac{\sin \pi t}{4\pi(\cosh^2(\pi (x-\xi )/2) -\cos^2(\pi t/2))}}\qquad (\zeta =x+it).\]
Also, $\{ z\in {\mathbb C}: \Im z>0\}$ is conformally equivalent to ${\mathbb D}$ by $z\mapsto (z-i)/(z+i)$, so $\zeta\mapsto \coth \pi (\zeta +i/2)/2$ gives a conformal map $\Sigma_1\rightarrow {\mathbb D}$. Benjamin obtained the periodic solution $u(x,t)=f(x-ct;p)$ where
   \[
   f(x;p)={\frac{(\Delta/2)\sinh p}{\cosh^2 (p/2)-\cos^2(\pi x/(2\ell))}},
   \]
which is proportional to the Poisson kernel for the strip $\{ x+ip: 0<x<\ell; -\infty <p<\infty \}$.
 The periodic and rational solutions discovered by Benjamin are related by formulas such as
   \begin{equation}\label{rationaltrig}
   \sum_{j=-\infty}^\infty {\frac{\eta}{(j-x)^2+\eta^2}}
     ={\frac{\pi \sinh 2\pi \eta}{ \pi(\sin^2\pi x+\sinh^2\pi\eta )}},
   \end{equation}
and variants of the Schwartz reflection principle.
\end{rem}

The next step is to replace a single travelling wave solution by a solution which is the sum of waves travelling at variable speeds, namely a periodic multi-soliton solution.

Let $(p_j,q_j)$ be canonically conjugate complex variables for the Hamiltonian dynamical system with Hamiltonian $K_n$ given in (\ref{KHam}).
This gives rise to the Hamilton--Jacobi equation
   \begin{equation}\label{HJ}
   {\frac{\partial S}{\partial t}}+{\frac{1}{2}}\sum_{j=1}^{n} \left( {\frac{\partial S}{\partial q_j}}\right)^2 +
   {\frac{k^2}{2}}\sum_{j,\ell =1 : j \neq l}^n{\hbox{cosec}}^2 {\frac{k(q_j-q_\ell )}{2}}=0.
\end{equation}
While the phase space of $K_n$ has dimension $2n$, we introduce a particular subspace of solutions that has dimension less than or equal to $n$. With $\phi (x)=-ik\cot (kx/2)$, a particular collection of solutions of the canonical equations of motion is given by
   \begin{equation}\label{q'again}
   {\frac{dq_\ell }{dt}}
       =-\phi (q_\ell -\bar q_\ell )-\sum_{m:m\neq \ell} \phi (q_m-q_\ell )-\sum_{m:m\neq \ell} \phi (q_\ell -\bar q_m),
   \end{equation}
so that
  \begin{equation}\label{q''}
  {\frac{d^2q_j}{dt^2}}
    ={\frac{-k^2}{2}}{\frac{\partial}{\partial q_j}}\sum_{m,\ell=1: m\neq \ell }^n{\hbox{cosec}}^2 {\frac{k(q_m-q_\ell)}{2}}.
  \end{equation}
(In (17) and (18) of \cite{Case} there are typographic errors which are corrected in the particular examples of that paper.) Suppose that $q_j(0)$ have $\Im q_j(0)>0$ so that $(e^{iq_1(0)}, \dots, e^{iq_n(0)})\in {\mathbb D}^n$ gives the initial condition of the Hamiltonian dynamical system, which we express as $\omega =n^{-1}\sum_{j=1}^n \delta_{e^{iq_j(t)}}$.

\begin{defn} (Coulomb model) Consider points $e^{iz_j}\in {\mathbb D}$ for $j=1,\dots ,n$, which represent the positions of unit positive electrical charges, subject to a continuous electric field $v: \mathbb{C}^2\rightarrow \mathbb{R}$; the constant $\beta$ is a scaling parameter.  Then the Hamiltonian is
   \begin{align}\label{electro}
     E_{n,v} &=\sum_{j=1}^n v(z_j, \bar z_j) -\beta \sum_{j=1}^n \log\vert \sin (k(z_j-\bar z_j)/2)\vert
                   -\beta\sum_{j,\ell =1: j\neq \ell }^n \log \vert \sin  (k(z_j- z_\ell )/2)\vert \nonumber\\
             &\qquad   -\beta\sum_{j,\ell =1: j\neq \ell }^n\log \vert \sin  (k(z_j- \bar z_\ell )/2)\vert ,
    \end{align}
for the canonical variables $(x_j,y_j)_{j=1}^n$, where $z_j=x_j+iy_j$.
 \end{defn}

\begin{prop}
\begin{enumerate}[(i)]
\item The canonical equations of $E_{n,v}$ give (\ref{q'again}) when $v=0$.
\item The Gibbs measure
   \begin{multline*}
   Z_n^{-1}
   \exp \left(-\sum_{j=1}^n \bigl( v(z_j, \bar z_j)-\beta \log \vert \sin (k(z_j-\bar z_j)/2)\vert \bigr) \right) \\
    \times \prod_{j,\ell =1: j\neq \ell}^n\Bigl( \vert e^{iz_j}-e^{iz_\ell}\vert \vert e^{iz_j}-e^{i\bar z_\ell}\vert \vert e^{i\bar z_j}-e^{iz_\ell}\vert \vert e^{i\bar z_j}-e^{i\bar z_\ell}\vert\Bigr)^{\beta/2}\prod_{j=1}^n dx_jdy_j
   \end{multline*}
is invariant under the flow generated by the canonical equations.
\end{enumerate}
\end{prop}

\begin{proof} (i) The points $e^{iz_j}$ and $e^{i\bar z_j}$ are conjugate with respect to the unit circle, and for each pair of distinct vertices $j,\ell,$ the configuration $e^{iz_j}, e^{i\bar z_j}, e^{iz_\ell }$ and $e^{i\bar z_\ell}$ gives $6$ connecting edges. The canonical equations are
  \begin{align*}
     {\frac{dx_j}{dt}}&={\frac{\partial E_{n,v} }{\partial y_j}}   \\
        &  ={\frac{\partial v}{\partial z}}(z_j, \bar z_j)i- {\frac{\partial v}{\partial \bar z}}(z_j, \bar z_j)i
              -{\frac{\beta k}{2}}\coth y_j \\
        &\qquad-{\frac{\beta k}{2}}\sum_{\ell =1: \ell\neq j}^n {\frac{\sinh k(y_j-y_\ell )}{\sin^2 (k(x_j-x_\ell )/2)+\sinh^2(k(y_j-y_\ell )/2)}}  \\
        &\qquad - {\frac{\beta k}{2}}\sum_{\ell =1: \ell\neq j}^n {\frac{\sinh k(y_j+y_\ell )}{\sin^2 (k(x_j-x_\ell )/2)+\sinh^2(k(y_j+y_\ell )/2)}},
  \end{align*}
and likewise
    \begin{align*}
        -{\frac{dy_j}{dt}}
          &={\frac{\partial E_{n,v}  }{\partial x_j}}  \\
          &={\frac{\partial v}{\partial z}}(z_j, \bar z_j)
             +{\frac{\partial v}{\partial \bar z}}(z_j, \bar z_j)
             -{\frac{\beta k}{2}} \sum_{\ell =1: \ell\neq j}^n 
             {\frac{\sin k(x_j-x_\ell )}{\sin^2 (k(x_j-x_\ell )/2)+\sinh^2(k(y_j-y_\ell )/2)}} \\
          &\qquad - {\frac{\beta k}{2}}\sum_{\ell =1: \ell\neq j}^n {\frac{\sin k(x_j-x_\ell )}{\sin^2 (k(x_j-x_\ell )/2)+\sinh^2(k(y_j+y_\ell )/2)}},
    \end{align*}
hence, after trigonometric reduction, we obtain
    \begin{align*}
    {\frac{dz_j}{dt}}
      &= -2i {\frac{\partial v}{\partial z}} (z_j, \bar z_j)+i\beta k \cot (k(\bar z_j-z_j)/2)
                +i\beta k\sum_{\ell=1: \ell\neq j} ^n\cot (k(\bar z_j-\bar z_\ell )/2) \\
      &\qquad+i\beta k\sum_{\ell =1: \ell\neq j}^n \cot (k(\bar z_j-z_\ell )/2).
    \end{align*}
(ii) The expression $e^{-E_{n,v}  }$ gives the stated formula after some reduction. The invariance follows from Liouville's theorem.
\end{proof}

In \cite[p. 236]{Alb}, the authors discuss  the Coulomb model and how this is related to the sine-Gordon model of quantum fields. The Benjamin--Ono equation involves $u^3$ in $H_\beta$, which is too singular for their transform method to be directly applicable; however, we can still use a version of their vortex equations.  A vortex at $z\in {\mathbb C}$ is represented by the unit point mass $\delta_z$, and one computes $\Delta^{-1/2}\delta_z$. 

The canonical equations of (\ref{HJ}) give dynamics on points $q_j\in {\mathbb D}$ hence on the measures $\omega_n =n^{-1}\sum_{j=1}^n \delta_{e^{iq_j(t)}}$ on ${\mathbb D}$ and  $\omega_n^*=n^{-1}\sum_{j=1}^n \delta_{e^{i\bar q_j(t)}}$ on $\{ z: \vert z\vert >1\}$. The operation of the Poisson kernel $P: C({\mathbb T}; {\mathbb R})\rightarrow C(\bar {\mathbb D}; {\mathbb R})$ has adjoint known as the balayage  ${\mathcal S}:{\mathcal P}(\bar {\mathbb D})\rightarrow {\mathcal P}({\mathbb T})$, and we can extend this operation by the method of images so  that ${\mathcal S}\omega_n^*={\mathcal S}\omega_n$. Thus we have a map ${\mathbb D}\rightarrow {\mathcal P}({\mathbb T}) :$ $(e^{iq_j})_{j=1}^n\mapsto {\mathcal S}\omega_n$. 

We wish to compute ${\mathcal S} \omega_n$ and its conjugate function.

\begin{lem}
The potential corresponding to ${\mathcal{S}}(\omega_n)$,
  \[
    w_n(x )=\int \log\vert e^{ix}-e^{i\phi }\vert \, {\mathcal{S}}(\omega_n) (\phi ) \, d\phi
  \]
evolves under the dynamics of (\ref{q'again}) so that $q_j=\xi_j+i\eta_j$ and
  \[
  u(x,t)={\mathcal H}{\frac{\partial }{\partial x}}\omega_n
        =-\sum_{j=1}^n 2k P_{e^{-k\eta_j}}(k(x-\xi_j)) .
  \]
\end{lem}

\begin{proof} We have
  \begin{equation*}
     w_n(x)= \int_{\mathbb D}\log \vert e^{ix} -z\vert \, \omega_n(dz)
       =\int \log\vert e^{ix}-e^{i\phi }\vert \, {\mathcal{S}}(\omega_n) (\phi ) \, d\phi
   \end{equation*}
with derivative
  \begin{equation*}
  {\frac{dw_n}{dx }}
     ={\hbox{p.v.}}\int_{\mathbb T} {\frac{1}{2}}\cot {\frac{x -\phi }{2}}\, {\mathcal{S}}(\omega_n )(\phi ) \, d\phi .
   \end{equation*}
so that ${\frac{dw_n}{dx}}$ belongs to $L^s$ for $0<s<1$. We can also form ${\frac{d^2w_n}{dx^2 }}$.
 Observe that $z\mapsto \log (e^{ix}-z )$ is holomorphic for $z\in {\mathbb D}$ with real part $\log \vert e^{ix}-z\vert$.
Let $q_j =\xi_j+i\eta_j$. Then 
   \begin{equation*}
   \int \log \vert e^{ix}-e^{i\phi}\vert  \, {\mathcal{S}}\left(\sum_{j=1}^n \delta_{e^{iq_j}}\right)(d\phi )
       = \sum_{j=1}^n \log\vert e^{ix}-e^{iq_j}\vert
   \end{equation*}
which has derivative
   \begin{align}\label{w}
   {\frac{\partial}{\partial x}}\sum_{j=1}^n \log \vert e^{ix}-e^{iq_j}\vert
      &=\sum_{j=1}^n {\frac{\sin (x-\xi_j)}{4\sin^2 \bigl((x-\xi_j)/2\bigr) +4\sinh^2 (\eta_j/2)}} \nonumber\\
      &={\frac{k}{4}}\sum_{j=1}^n \left( \cot {\frac{k}{2}}(x-\xi_j+i\eta_j) +\cot {\frac{k}{2}}(x-\xi_j-i\eta_j) \right).
   \end{align}
Let $q_j(t)=\xi_j(t)+i\eta_j(t)$ and then consider $x-q_j(t)$. Then the harmonic conjugate to (\ref{w}) is
   \begin{align}
     -u(x,t)
       &=\sum_{j=1}^n \left(-ik\cot {\frac{k(x-\xi_j-i\eta_j)}{2}}+ik\cot {\frac{k(x-\xi_j+i\eta_j)}{2}}\right)\nonumber\\
       &=\sum_{j=1}^n {\frac{k \sinh k\eta_j}{\sin^2 \bigl( k( x-\xi_j)/2\bigr) +\sinh^2(k\eta_j/2)}}\nonumber\\
       &=\sum_{j=1}^n 2k P_{e^{-k\eta_j}}(k(x-\xi_j))
,\end{align}
in terms of the Poisson kernels for ${\mathbb D}$.
\end{proof}

\begin{prop}[Case \cite{Case}]\label{Caseprop} Let $(e^{iq_j(t)})_{j=1}^n\in {\mathbb D}^n$ evolve according to the canonical equations (\ref{q'}), and suppose that
$\eta_j(0)>0 $. Then
   \begin{equation*}
   u(x,t)=\sum_{j=1}^n {\frac{-k\sinh k\eta_j}{\sinh^2 (k\eta_j/2) +\sin^2 (k(x-\xi_j)/2)}}
   \end{equation*}
gives a solution of the Benjamin--Ono equation which is periodic in the $x$-variable.
\end{prop}

\begin{proof} See \cite{Case} or page 203 of \cite{Ablowitz}.
\end{proof}

It follows that there is a family of multi-soliton solutions of the Benjamin--Ono equations, parameterised by the initial conditions $\{ \omega +i {\mathcal H}\omega \}$,
and via the map ${\mathbb D}^n\rightarrow {\mathcal P}({\mathbb T}) $ : $(e^{iq_j})_{j=1}^n\mapsto {\mathcal S}\omega_n$, so we can use $\omega_n\in{\mathcal P}({\mathbb T})$ as a system of parameters for the $n$-solitons.
 In the following two sections, we consider these solutions as $n\rightarrow\infty$, and we do this in two ways, according to the Hamiltonian system governing the $(e^{iq_j})_{j=1}^n$.
In Section~\ref{S:RandomM}, we consider $E_{n,v}$ and the associated Gibbs measures as $n\rightarrow\infty$.  In Section~\ref{S:ContHS}, we show that in a suitable sense $K_n\rightarrow K$, where $K$ is the Hamiltonian of the isentropic Euler equations.

The soliton solutions may be expressed in terms of elliptic functions $\wp$, which are rational on the complex torus, hence are determined by the positions of their poles and zeros. In our case, $\wp$ degenerates to ${\hbox{cosec}}^2 k(x-y)/2$, and the poles satisfy some identities which are instances of the addition rules for
$\wp$; see \cite{BullCaud}. In Section~\ref{S:RandomM}, we show how these solitons arise from a Coulomb gas in the plane; models of this kind have previously been found for quantum fields as in \cite{Alb}.

\section{Random matrix model}\label{S:RandomM}

In the previous section, we introduced a Gibbs measure $E_{n,v}$  and a related soliton solution of the Benjamin--Ono equation. In this section, we consider the Gibbs measure as $n\rightarrow\infty$. We show that the potential $v$ determines a probability measure $\nu_n$ on the space $U(n)$ of $n\times n$ unitary matrices, and hence a probability measure $\tilde\nu_n$ on the maximal torus ${\mathbb T}^n$. Under  various conditions on $v$, we show that $\nu_n$ and $\tilde\nu_n$ satisfy concentration of measure phenomena as $n\rightarrow\infty$. Several of the arguments in this section are well known, but they are not necessarily familiar in the context of PDE, so we include them for completeness.

Taking a continuous $v:[0, 2\pi ]\rightarrow {\mathbb R}$ as the starting point,
we let ${\tilde E}_{n,v}$ be the function defined by
  \[
  {\tilde E}_{n,v} (\theta_1, \dots, \theta_n) = {\frac{1}{n}}\sum_{j=1}^n{ v(\theta_j) } - {\frac{1}{n^2}}\sum_{1 \leq j < k \leq n}{ \log \Bigl( 4 \sin^2 \frac{1}{2}\left( \theta_k - \theta_j \right) \Bigr) }.
   \]
We can obtain $\tilde E_{n,v} $ from (\ref{electro}) by allowing $e^{iz_j}$ to approach the unit circle, and adjusting the scalar potential $v(z_j, \bar z_j)$ to remove the terms $\cot (k(\bar z_j-z_j)/2)$. Thus we obtain a model for a Coulomb gas of $n$ atoms on the unit circle, subject to an electric field $v$. Then we introduce a probability measure $\tilde \nu$ on ${\mathbb T}^n$ by
   \begin{align}\label{eigdist}
   {\tilde \nu}_n(d\Theta )
     &= Z_n^{-1}e^{-n^2 {\tilde E}_{n,v} }{\frac{d\theta_1}{2\pi}}\dots  {\frac{d\theta_n}{2\pi}}\nonumber\\
     &=Z_n^{-1} e^{-\sum_{j=1}^n{ nv(\theta_j) }}
          \prod_{1 \leq j < k \leq n}{4 \sin^2 \frac{1}{2}\left( \theta_k - \theta_j \right) }
              \, {\frac{d\theta_1}{2\pi}}\dots  {\frac{d\theta_n}{2\pi}},
\end{align}
where $Z_n$ is a normalising constant.

In order to express $Z_n$ as a Toeplitz determinant, impose the condition $\beta=2$, and then use \emph{Andr\'eief's Identity}.
For $j \in \{ 1, \, \ldots \,, n \}$ let $f_j$ and $g_j$ be continuous complex functions defined on $[0, 2\pi]$. Then,
\begin{multline}
\det \left[ \int_{[0, 2\pi]}{ f_j(x) g_k(x) \,dx } \right]_{j,k = 1}^n  \\
    = \frac{1}{n!}\int \cdots \int_{[0, 2\pi]^n}{ \det\left[ f_j(x_\ell{}) \right]_{j,\ell = 1}^n
                   \det\left[ g_k(x_\ell{}) \right]_{k,\ell = 1}^n \,dx_1 \ldots dx_n }.
\end{multline}

\begin{prop}\label{Toeplitzprop}
Let $v:[0, 2\pi ]\rightarrow {\mathbb C}$ be a continuous function. Then
   \[
       \int \cdots \int_{[0, 2\pi]^n}{ e^{-n^2 {\tilde E}_{n,v} } \, d\theta_1 \ldots d\theta_n}
           = n!\det\left[ \int_{[0, 2\pi]}{ e^{i(j-k)\theta - nv(\theta)} \, d\theta} \right]_{j,k = 1}^n.
    \]
\end{prop}

\begin{proof}
First note that $\vert{e^{i\theta_k} - e^{i\theta_j} }\vert^2= 4\sin^2\left( \frac{\theta_k - \theta_j}{2} \right)$, by the double angle formula. Therefore,
   \[
     e^{-n^2 {\tilde E}_{n,v} }
         = e^{-\sum_{j=1}^n{ nv(\theta_j) }} \prod_{1 \leq j < k \leq n}{ \vert{ e^{i\theta_k} - e^{i\theta_j} }\vert^2  }.
   \]
We observe that the final factor looks like a Vandermonde determinant. We have
   \begin{align*}
     & \int \cdots \int_{[0, 2\pi]^n}{ e^{-n^2 {\tilde E}_{n,v}} \, d\theta_1 \ldots d\theta_n}\nonumber \\
     & \quad = \int \cdots \int_{[0, 2\pi]^n}{ \det\left[ {\hbox{diag}}\left( e^{-nv(\theta_j)}\right)_{j=1}^n \right] \det\left[ e^{i(j-1)\theta_\ell{}} \right]_{j,\ell = 1}^n \det\left[ e^{-i(k-1)\theta_\ell{}} \right]_{k,\ell = 1}^n \, d\theta_1 \ldots d\theta_n} \nonumber\\
     & \quad =  \int \cdots \int_{[0, 2\pi]^n}{ \det\left[ e^{ij\theta_\ell{}} \right]_{j,\ell = 1}^n \det\left[ e^{-ik\theta_\ell{} - nv(\theta_\ell{})} \right]_{k,\ell = 1}^n \, d\theta_1 \ldots d\theta_n}.
   \end{align*}
Finally, an application of Andr\'eief's Identity yields,
   \[
       \int \cdots \int_{[0, 2\pi]^n}{ e^{-n^2 {\tilde E}_{n,v}} \, d\theta_1 \ldots d\theta_n}
          = n!\det\left[ \int_{[0, 2\pi]}{ e^{i(j-k)\theta - nv(\theta)} \, d\theta} \right]_{j,k = 1}^n.
  \]
\end{proof}

We can regard ${\mathcal T}^n$ as the maximal torus in the compact Lie group $U(n)$ of $n\times n$ unitary matrices, and recall that every $X \in U(n)$ is conjugate to some element of ${\mathcal T}^n$, so there is a map $\Lambda :U(n)\rightarrow {\mathcal T}^n: X \mapsto (e^{i\theta_1}, \dots, e^{i\theta_n})$. We associate with  $(e^{i\theta_1}, \dots, e^{i\theta_n})$ the empirical measure $n^{-1}\sum_{j=1}^n \delta_{e^{i\theta_j}}$, so composing these gives the map
   $ U(n)\rightarrow {\mathcal P}({\mathbb T})$, $X \mapsto  \omega^X = n^{-1}\sum_{j=1}^n \delta_{e^{i\theta_j}}$.

Let $\mu_n$ be the Haar probability measure on $U(n)$, and consider the function
   \[
   \psi_\beta (\theta_1, \dots, \theta_n)
      =\left\vert \prod_{1\leq j<\ell \leq n} \sin {\frac{\pi( \theta_j-\theta_\ell )}{L}}\right\vert^\beta,
   \]
which we describe first in terms of random matrices, and then in (\ref{tdS}) in terms of the ground state of a Schr\"odinger equation. First,  for $\beta=2$, and $L=2\pi$,  any continuous class function $F:U(n)\rightarrow {\mathbb C}$ satisfies
   \[
      \int_{U(n)} F(X) \, \mu_n(dX)
      ={\frac{2^{n(n-1)}}{n!}}\int_{{\mathcal T}^n} F(e^{i\theta_1}, \dots , e^{i\theta_n}) \psi_2 (\theta_1, \dots, \theta_n)
                \prod_{j=1}^n {\frac{d\theta_j}{2\pi} }.
   \]
The map $\Lambda :U(n)\rightarrow {\mathcal T}^n$ therefore induces the probability measure on ${\mathcal T}^n$  with density $\psi_2$.

Then we recognise (\ref{eigdist}) as the probability measure on ${\mathcal T}^n$ that is induced by $\Lambda$ from the probability measure
   \begin{equation}\label{matrixdist}
   \nu_n(dX)
      =\tilde Z_n^{-1}\exp (-n\, {\hbox{trace}}\, Q(X))\, \mu_n (dX)
    \end{equation}
on $U(n)$.

Let ${\mathcal L}^2(n)$ be the space $M_n$ of complex matrices with the normalised Hilbert--Schmidt norm
$\Vert X \Vert_{{\mathcal L}^2(n)}=({\hbox{trace}}X^\dagger X /n)^{1/2}$.
Let $\nabla $ be the invariant gradient operator over $U(n)$; let $V(X)=n \,{\hbox{trace}}\,Q(X)$ and $\Delta$ be the invariant Laplace operator on $U(n)$; then we let $L=\Delta -\nabla V\cdot \nabla$. The {\sl carr\'e du champ it\'er\'e} operator
associated with $L$ is $\Gamma_2$, as in \cite[p.~383]{V2009}, where
   \begin{equation}\label{Gamma2}
   \Gamma_2 (f )
      =\Vert {\hbox{Hess}}f\Vert^2_{{\mathcal L}^2} +\bigl( {\hbox{Ric}} +{\hbox{Hess}} V\bigr) (\nabla f , \nabla f ).
   \end{equation}

\begin{prop}
Suppose as in (\ref{freecurv}) that $v''(\theta)\geq \kappa -1/2$ for some $\kappa >0$. Then there exists $\kappa_1>0$ such that
  \begin{equation}\label{logsobunit}
     \int_{U(n)} F(X)^2 \log \Bigl(F(X)^2/\int F^2 \, d\nu_n\Bigr) \, \nu_n(dX)
         \leq {\frac{2}{\kappa_1n}}\int_{U(n)}\Vert \nabla F(X)\Vert^2 \, \nu_n(dX)
  \end{equation}
 with a constant $\kappa_1n$ which improves with increasing $n$.
\end{prop}

\begin{proof} We compute the Hessian term in (\ref{Gamma2}). The left translation operation gives a canonical means of transporting elements around $U(n)$, so we can carry out calculations of invariant operators on the identity element using the Lie algebra
$M^{sa}_n=\{ A\in M_n({\mathbb C}): A=A^\dagger\}$. Note that $M_n^{sa}$ is a real vector space of dimension $n^2$.
Let $V(X)=n \,{\hbox{trace}}\,Q(e^{iX})$,  where $Q(e^{i\theta })=\sum_j a_j e^{ij\theta }$.

Fix $X, Y \in U(n)$. Let $(\xi_k)_{k=1}^n$  be an orthonormal basis in ${\mathbb C}^n$
of eigenvectors of $X$ 
and let $y_{\ell ,k}=\langle Y \xi_k, \xi_\ell \rangle_{{\mathbb C}^n}$ for $k,\ell = 1,2,\dots,n$. Then one obtains a version of the Rayleigh--Schr\"odinger formula; by applying Duhamel's formula twice, one computes
   \begin{multline*}
     V(X+\varepsilon Y)  = V(X)
       + n\, {\hbox{trace}}\sum_ja_jij\varepsilon\int_0^1\Bigl(e^{ij(1-s)X}    \\
      +\int_0^1e^{ij(1-s)(1-u)(X+\varepsilon Y)} ij\varepsilon (1-s)Ye^{i(1-s)uX}du\Bigr)Ye^{isjX}\, ds,
    \end{multline*}
and hence
   \begin{align*}
   \bigl\langle {\hbox{Hess}} V, & \,Y\otimes Y\bigr\rangle
   =\lim_{\varepsilon\rightarrow 0}\varepsilon^{-2}\bigl( V(X+\varepsilon Y)+V(X-\varepsilon Y)-2V(X)\bigr)\nonumber\\
   &=n\sum_{j}\sum_{k=1}^n-j^2a_j \int_0^1\int_0^1e^{ij(1-u)(1-s)\theta_k+ijs\theta_k}\langle Ye^{iju(1-s)X} Y\xi_k,\xi_k\rangle_{{\mathbb C}^n} \, du\,ds\nonumber\\
   &= n\sum_{j}\sum_{k=1, \ell=1}^n-j^2a_j\int_0^1\int_0^1 e^{ij(1-u)(1-s)\theta_k+ijs\theta_k+iju(1-s)\theta_m}\vert y_{\ell,k}\vert^2
     \, du\,ds\nonumber\\
   &=\sum_{\ell=1}^n nv''(\theta_\ell) \vert\langle Y\xi_\ell,\xi_\ell \rangle\vert^2\nonumber\\
   &\qquad +n\sum_{1\leq k, \ell \leq n: k\neq \ell} \int_0^1
        {\frac{v'((1-s)\theta_\ell+s\theta_k)-v'(\theta_k)}{(1-s)(\theta_\ell-\theta_k )}}
       \, ds \, \vert\langle Y\xi_k,\xi_\ell \rangle\vert^2,
   \end{align*}
where $v(\theta) =Q(e^{i\theta })$.

The underlying idea of the proof is that there exists $\kappa_2>0$ such that ${\hbox{Ric}}(SU(n))\geq \kappa_2nI,$ so by taking $\kappa_1>0$ small enough, we can ensure that
   \begin{equation*}
   \Gamma_2(F)\geq \kappa_1n\Vert \nabla F\Vert_{{ L}^2}^2,
   \end{equation*}
as in the Bakry--{\'E}mery condition. The Bakry--{\'E}mery theorem leads to a logarithmic Sobolev inequality for scalar-valued functions
 as in \cite[p.~547]{V2009}.
\end{proof}

To deduce results about the eigenvalue distribution, we need two facts. The first is that for class functions, the right-hand side of (\ref{logsobunit}) reduces to
   \[
   \int_{U(n)}\Vert \nabla F(X)\Vert^2 \, \nu_n(dX)
      =\int_{{\mathcal T}^n} \Vert\nabla F(\Theta) \Vert^2 \psi_2 (\Theta )e^{-V(\Theta) } \, d\Theta ,
   \]
where $V(\Theta) =n\sum_{j=1}^n v(\theta_j)$
and $d\Theta=\prod_{j=1}^n (d\theta_j/(2\pi ))$. 
For $X \in U(n)$, the normalised eigenvalue counting function  is $N^X_n : [0,2\pi] \to \mathbb{R}$,
   \[N^X_n(\theta )= \frac{1}{n} \, \sharp \{ j: \theta_j\leq \theta \} = \int_{[0, \theta ]} \omega^X (d\phi ) \]
where $(e^{i\theta_j})_{j=1}^n$ are the eigenvalues of $X\in U(n)$, listed according to multiplicity with $0\leq \theta_1\leq\theta_2\leq\dots\leq \theta_n\leq 2\pi$.

\begin{lem}\label{Lip}
There exists $\kappa_3>0$ independent of $n$ such that the map $X\mapsto \omega^X$ is $\kappa_3$-Lipschitz 
$(U(n), {\mathcal L}^2(n))\rightarrow  ({\mathcal P}({\mathbb T}), W_2)$.
\end{lem}

\begin{proof} By a version of Lidskii's formula, the map $X\mapsto (e^{i\theta_j})_{j=1}^n$ is $\kappa_3$-Lipschitz 
   \[ (U(n), {\mathcal L}^2(n))\rightarrow ({\mathcal T}^n, \ell^2(n))\]
for some absolute $\kappa_3>0$. By a duality argument (see \cite{Bl2001,HPU})  one shows that the map $({\mathcal T}^n, \ell^2(n))\rightarrow ({\mathcal P}({\mathbb T}), W_2)$ is $1$-Lipschitz.
Hence the map $X\mapsto \omega^X$ is $\kappa_3$-Lipschitz from $(U(n), {\mathcal L}^2(n))$ to $({\mathcal P}({\mathbb T}), W_2)$.
\end{proof}

Starting with $Q\in C^2({\mathbb T};{\mathbb R})$, we can introduce an equilibrium density $\rho_0\in {\mathcal P}({\mathbb T})$, and we now consider what is meant by multi-soliton solutions of the Benjamin--Ono equation with initial vortices distributed according to $\rho_0$. Fix $k=1$ and $\eta >0$, and suppose that $\eta_j\geq \eta$ for all $j=1, \dots, n$. Define $f: {\mathcal T}^n \to \mathbb{R}$
by
  \[
     f(\theta_1, \dots, \theta_n)={\frac{1}{n}}\sum_{j=1}^n {\frac{\sinh \eta_j}{\sinh^2(\eta_j /2 )+\sin^2 ((x-\theta_j)/2)}}.
  \]
For comparison, $nf(\theta_1, \dots, \theta_n)$ is the expression which appears in $u(x,0)$ from Proposition \ref{Caseprop} for the initial condition for the solution to the Benjamin--Ono equation.

\begin{cor}\label{linstat}
There exists an absolute constant $\kappa_4>0$ such that for all $n$ and all $s > 0$
   \begin{equation}\label{conc}
   {\tilde \nu}_n\Bigl\{ \Theta\in {\mathcal T}^n: \Bigl\vert f(\Theta )-\int fd\tilde \nu_n\Bigr\vert \geq {\frac{s}{\sqrt{n}}}\Bigr\}
          \leq 2e^{-s^2/4\kappa_4}.
  \end{equation}
\end{cor}

\begin{proof} By elementary estimates,
    \begin{align*}
  -{\frac{\partial}{\partial\theta}} {\frac{\sinh \eta}{\sinh^2(\eta/2)+\sin^2(\theta/2)}}
   &={\frac{2\sinh (\eta/2) \sin (\theta/2)}{\sinh^2(\eta/2)+\sin^2(\theta/2)}}
          {\frac{ \cosh (\eta/2)\cos (\theta/2)}{\sinh^2(\eta/2)+\sin^2(\theta/2)}}  \\
   &\leq {\frac{\cosh (\eta/2)}{\sinh^2(\eta/2)}}  \\
   &\leq {\frac{4}{(1-e^{-\eta})^2}}\qquad (\eta>0).
 \end{align*} 
We therefore have
$\Vert \nabla f\Vert_{\ell^2(n)}  \leq 4/(1-e^{-\eta})^2$, so $f$ defines a Lipschitz function on ${\mathcal T}^n$.
Define the class function $F: U(n) \to \mathbb{C}$ by $F(X) = f(\Lambda(X))$.

 As in Lemma \ref{Lip}, we deduce that
   \[
   \vert F_n(X_n)-F_n(Y_n)\vert \leq {\frac{\kappa_3}{(1-e^{-\eta })^2}}\Vert X_n-Y_n\Vert_{{\mathcal L}^2(n)}\qquad (X_n,Y_n\in U(n)).\]
By the logarithmic Sobolev inequality (\ref{logsobunit}) with the advantageous constant $n$, we have
   \[
   \int_{U(n)}\exp\left( tF_n(X_n) -t\int F_n d\nu_{n} \right) \nu_n(dX)
       \leq  \exp\left( {\frac{\kappa_4t^2}{n}} \right)\qquad (t\in {\mathbb R}),
   \]
for some $\kappa_4>0$ independent of $n$, so (\ref{conc}) follows by Chebyshev's inequality.
\end{proof}

\begin{rem} Let $Q=0$ and suppose that $\eta_j=\eta>0$ for all $j$, and let
   \[
       g(\theta) = {\frac{\sinh \eta}{\sinh^2(\eta/2)+\sin^2((x-\theta )/2)}}-2
   \]
so that $\int_0^{2\pi} g(\theta) d\theta=0$ and $g\in H^{1/2}$. Then by Corollary 2.3 of \cite{Jo}, the random variables
   \[
      {\hbox{trace}}\, g(X_n)=\sum_{j=1}^n g(\theta_j )
   \]
for $(X_n)_{n=1}^\infty\in  \prod_{n=1}^\infty (U(n), \mu_n)$, converge in distribution to a normal random variable with mean $0$ and variance $2\Vert g\Vert_{H^{1/2}}^2$ as $n\rightarrow\infty$. Remarkably,  this version of the central limit theorem does not involve any scaling constant such as $1/\sqrt{n}$ in Corollary \ref{linstat}.
\end{rem}

We now change perspective, and start with densities $\rho$. For every probability density function $\rho$ on ${\mathbb T}$ with $\rho\in L^3({\mathbb T})$, the field
   \[
   Q(e^{i\theta})
      ={\mathcal L}\rho (\theta)=2\int_{\mathbb T} \log \vert e^{i\theta}-e^{i\phi}\vert \, \rho (\phi )\, d\phi
   \]
has derivative
   \[
   {\frac{d}{d\theta}}Q(e^{i\theta})={\hbox{p.v.}}\int_{\mathbb T} \cot {\frac{\theta-\phi}{2}} \rho(\phi) \, d\phi
       ={\mathcal H}{\rho}(\theta ),
   \]
so we can recover $\rho_0$  from $v'(\theta )={\frac{d}{d\theta}}Q(e^{i\theta})$ via this singular integral equation. This $v$ is called the scalar potential.

 We return to the context of (\ref{defQ}) under the condition (\ref{freecurv}). Suppose that the support of $\omega$ is contained in the support of $\rho_0$. Then by \cite{HPU},
   \[
   {\tilde \Sigma}_Q(\omega )
     =\int\!\!\!\int_{{\mathbb T}^2} \bigl(\omega (\theta )-\rho_0(\theta )\bigr)
          \bigl(\omega (\phi)-\rho_0(\phi)\bigr)\log{\frac{1}{\vert e^{i\theta}-e^{i\phi}\vert}}\, d\theta d\phi.
   \]
When we operate on the probability measures with the Poisson kernel, replacing $\omega$ by $P_r\omega$ and $\rho_0$ by $P_r\rho_0$, we make the double integral on the right-hand side smaller.

Hiai, Petz and Ueda \cite[Theorem~1.1]{HPU} show that
   \begin{equation}\label{HPUformula}
     {\tilde \Sigma}_Q(\omega )
       =\int_{\mathbb T} Q(e^{it})\, \omega (dt)
          +\int\!\!\!\int_{{\mathbb T}^2\setminus\Delta}\log {\frac{1}{\vert e^{i\psi}-e^{it}\vert}}\, \omega (d\psi )\omega (dt)+B(Q),
   \end{equation}
where $\Delta =\{ (\theta, \phi )\in {\mathbb T}^2:\theta =\phi \}$ and
    \begin{equation}\label{BQformula}
    B(Q)=\lim_{n\to\infty}{\frac{1}{n^2}}
            \log\left( \int\!\! \dots\!\int_{{\mathcal T}^n} \exp\left({-\sum_{j=1}^n{ nQ(e^{i\theta_j})}}\right)
             \prod_{1 \leq j < k \leq n}{ \vert{ e^{i\theta_k} - e^{i\theta_j} }\vert^2\, d\theta_1\dots d\theta_n }\right).
    \end{equation}
The advantage of this expression over (\ref{deffreeent}) is that  (\ref{HPUformula}) can be computed without prior knowledge of the minimiser, and the integral in  (\ref{BQformula}) is otherwise given by the Toeplitz determinant in Proposition~\ref{Toeplitzprop}, which make standard asymptotic formulas available.
Let $V(\theta )=Q(e^{i\theta})$, and suppose that $V$ is $C^2$.

\indent Hiai, Ueda and Petz \cite{HPU} also show that the Wasserstein transportation cost function $c(\theta , \phi )=2^{-1}\vert e^{i\theta}-e^{i\phi}\vert^2$ satisfies
\begin{equation}\label{freetrans}W_2(\mu, \rho_0)^2\leq {\frac{2}{1-2\kappa_1}}\tilde\Sigma_Q(\mu )\qquad (\mu \in {\mathcal P}).\end{equation}
This is the free transportation inequality (\ref{FTI}), which is sharper than the $T_2$ transportation inequality involving the classical relative entropy.
Hiai, Petz and Ueda  used  the concentration of measure phenomenon to show that the empirical distributions $\omega_n$ converge weakly almost surely under $\tilde\nu_n$ to $\rho_0 (\theta)d\theta$ as $n\rightarrow\infty$.

We now explain how the free measure on $H^{1/2}$ arises. The tangent space to ${\mathcal T}^n$ at $(e^{i\theta_1}, \dots, e^{i\theta_n})$ may be identified with $L^2 (\omega_n)$, so we can compute
  \begin{align*}
  \bigl\langle &{\hbox{Hess}}\tilde E_{n,V} , f\otimes g\bigr\rangle      \\
  &= {\frac{1}{n}}\sum_{j=1}^n V''(\theta_j)f(e^{i\theta_j})g(e^{i\theta_j})+{\frac{1}{n^2}}\sum_{1\leq j,\ell \leq n: j\neq \ell} \frac{(f(e^{i\theta_j})-f(e^{i\theta_\ell}))(g(e^{i\theta_j})-g(e^{i\theta_\ell}))}{\vert e^{i\theta_j}-e^{i\theta_\ell}\vert^2}      \\
  &=\int v''(e^{i\theta}) f(e^{i\theta})g(e^{i\theta}) \, \omega_n(d\theta)    \\
  &\qquad+\iint_{[\theta\neq \phi]} \frac{(f(e^{i\theta})-f(e^{i\phi}))(g(e^{i\theta})-g(e^{i\phi}))}{\vert e^{i\theta}-e^{i\phi}\vert^2} \, \omega_n(d\theta) \omega_n(d\phi).
  \end{align*}
For $v=0$ and $\omega (d\theta )=d\theta/2\pi$, the final bilinear form is the inner product of $H^{1/2}$. Observe also that
   \[
   {\hbox{trace}}\, {\hbox{Hess}}\tilde E_{n,V}={\frac{1}{n}}\sum_{j=1}^nV''(\theta_j)+{\frac{1}{4n^2}}\sum_{1\leq j,\ell \leq n: j\neq \ell}{\hbox{cosec}}^2 {\frac{\theta_j-\theta_\ell}{2}},
   \]
which resembles $K_n$.

While (\ref{KHam}) is a classical Hamiltonian system, the corresponding quantum Hamiltonian,
\begin{equation}\label{tdS}i{\frac{\partial \Phi}{\partial t}}=-{\frac{1}{2}}\sum_{j=1}^n {\frac{\partial^2 \Phi}{\partial\theta_j^2}} +{\frac{\pi^2\beta  (\beta -1)}{L^2}}\left(\sum_{1\leq j<\ell \leq n} {\hbox{cosec}}^2 {\frac{\pi( \theta_j-\theta_\ell ) }{L}}\right) \Phi -E\Phi \end{equation}
is also integrable, and the ground state of the corresponding time-independent equation has a very special form like $\psi_2(\Theta )$.  
For $n\geq 5$, each summand ${\hbox{cosec}}^2 {\frac{k(q_j-q_\ell )}{2}}$ belongs to $L^2 ({\mathcal T}^n; {\mathbb C})$, so (\ref{tdS}) is densely defined. 
 Suppose that $\Phi (\Theta) =\psi (\Theta )e^{iS(\Theta )}$ is a solution to (\ref{tdS}) where $\psi, S: {\mathcal T}^n\rightarrow {\mathbb R}$. Then the phase satisfies
\begin{equation} {\frac{\partial S}{\partial t}} +{\frac{1}{2}}\nabla S\cdot\nabla S+{\frac{\pi^2\beta  (\beta -1)}{L^2}}\left(\sum_{1\leq j<\ell \leq n} {\hbox{cosec}}^2 {\frac{\pi( \theta_j-\theta_\ell ) }{L}}\right) -E=0.\end{equation} 
In the next section, we introduce a Hamiltonian system that has canonical equations
$${\frac{\partial q}{\partial t}}+{\frac{1}{2}}\left({\frac{\partial  q}{\partial x}}\right)^2+{\frac{\kappa_U }{2}}\rho_t(x)^2=0,$$
and we can write this in terms of the potential by $\rho ={\mathcal H}{\frac{\partial Q}{\partial x}}$.  
For comparison, if 
\begin{equation} {\frac{\partial q}{\partial t}}+\beta \left(  {\frac{\partial q}{\partial x}}\right)^2+{\mathcal H}{\frac{\partial^2 q}{\partial x^2}}=0;\end{equation}
then $u= {\frac{\partial q}{\partial x}}$ satisfies the Benjamin--Ono equation (\ref{BO-eqn}). 

\section{Continuous Hamiltonian systems}\label{S:ContHS}

We proceed to show that the isentropic Euler equations arise as the limit of (\ref{q'}) and the canonical equations as $n\rightarrow\infty$, under a suitable scaling. Suppose that (\ref{q'}) is real-valued. Then, passing to the limit $n\rightarrow\infty$, we can consider a PDE
  \[ {\frac{\partial q}{\partial t}} =\xi (q(x,t);t)  \] 
and the evolution $q(x,0)\mapsto q(x,t)$, with the corresponding induced map on probability measures $\rho (dy;0)\mapsto \rho (dy;t)$ where
   \[ \int f(y)\, \rho (dy;t)=\int f(q(x,t)) \,\rho (dx;0) \]
which is the counterpart of  (\ref{q'}).

Note that for any positive continuous function $\rho_0:{\mathbb T}\rightarrow {\mathbb R}$, we have
   \[
   {\frac{1}{4\pi}}\int_{\mathbb T} {\frac{\sinh \eta}{\sinh^2(\eta /2) +\sin^2 ((x-y)/2)}}\rho_0(y)\, dy
      \rightarrow \rho_0(x)\qquad (\eta \rightarrow 0+),
   \]
and
   \[
   {\frac{1}{2n}}\int_{\vert x-y\vert>1/(2n\rho_0(x))} {\hbox{cosec}}^2 (x-y)\, \rho_0(y)\, dy
     \rightarrow 2\rho_0(x)^2\qquad (n\rightarrow\infty ).
   \]

Let $F_n(\lambda )=\int_0^\lambda \omega_n(dx)$ be the cumulative distribution function of $\omega_n$. Then
   \[
      W_2^2 (\omega_n, \omega_m)=\int \vert F_n(x)-F_m(x)\vert^2\, dx.
   \]
We replace $\omega_n$ by $\tilde\omega_n=\omega_n\ast P_{1-2^{-n}}$ where $P_r$ is the Poisson kernel, so $\tilde\omega_n$ has cumulative distribution
function $\tilde F_n$. Then $\tilde F_n$  has a strictly positive and continuous density, hence  $\tilde F_n$ has continuously differentiable inverse  function $\varphi_n$, so that $\tilde \omega_n$ is the probability measure induced from $( [0,1], dx)$ by $\varphi_n$. Then $\tilde F_n'(\varphi_n(x))\varphi_n'(x)=1$, so we can approximate
    \[
    \int\!\!\!\int_{[0,2\pi ]\setminus \Delta} {\hbox{cosec}}^2 {\frac{x-y}{2}} \, \omega_n(dx)\omega_n(dy)
       = \int\!\!\!\int_{[0,2\pi ]\setminus \Delta} {\hbox{cosec}}^2 {\frac{\varphi_n(x)-\varphi_n(y)}{2}} \, dxdy
       \]
in which $\varphi_n(j/n)-\varphi_n(\ell /n)=(j-\ell )/(n\tilde F_n'(\varphi_n(\bar x)))$, for some $\bar x$ between $j/n$ and $\ell/n$.

Suppose  that with $k=1/n$, the empirical distribution $\omega^x_n=n^{-1}\sum_{j=1}^n \delta_{kq_j}$ converges weakly to a probability measure with probability density function $\rho$ on the circle as $n\rightarrow\infty$. By some simple estimates on the Poisson kernel, we have $W_2^2 (P_r\ast \omega_n, \omega_n)\leq C\sqrt{1-r}$ for some absolute constant $C$, so with $r=1-2^{-n}$, we have $\tilde\omega_n\rightarrow \rho (x)\, dx$ likewise as $n\rightarrow\infty$. Then with $q_j/n=\varphi_n (j/n)$,
   \begin{align*}
    {\frac{1}{2n^2}}\sum_{\ell =1: \ell \neq j}^n{\hbox{cosec}}^2 {\frac{(j-\ell )}{2n\rho (q_j/n)}}
      &= {\frac{1}{2n^2}}\sum_{\ell =1: \ell \neq j}^n\left({\hbox{cosec}}^2 {\frac{(j-\ell )}{2n\rho (q_j/n)}} -\left({\frac{(j-\ell )}{2n\rho (q_j/n)}}\right)^{-2}\right)\\
      &\qquad\qquad+\sum_{\ell =1: \ell \neq j}^n{\frac{2\rho (q_j/n)^2}{(j-\ell )^2}}\\
      &=2\rho (q_j/n)^2\left( {\frac{\pi^2}{3}}+O\left({\frac{1}{j}}\right)
              +O\left({\frac{1}{n-j}}\right)\right)  \quad (j=1, \dots, n),
    \end{align*}
so summing over $j$, we have
   \[
   {\frac{1}{2n^3}}\sum_{j,\ell =1: \ell \neq j}^n{\hbox{cosec}}^2 {\frac{(j-\ell )}{2n\rho (q_j/n)}}
     \rightarrow {\frac{2\pi^2}{3}}\int_{\mathbb T} \rho (x )^3 dx .
   \]
Likewise, we have
   \begin{align*}
   {\frac{dq_\ell}{dt}}
       &=ik\sum_{m:m\neq \ell}\left( \cot {\frac{k(q_\ell -q_m)}{2}}
            +\cot {\frac{k(\bar q_m-q_\ell)}{2}}\right) +ik\cot {\frac{k(\bar q_\ell -q_\ell)}{2}}\nonumber\\
        &=ik\sum_{m:m\neq \ell} {\frac{\sin {\frac{k(\bar q_m -q_m)}{2}}}
                     {\sin {\frac{k(q_\ell -q_m)}{2}}
                     \sin {\frac{k(\bar q_m -q_\ell)}{2}}}} + ik\cot {\frac{k(\bar q_\ell -q_\ell)}{2}},
  \end{align*}
and with $q_\ell-q_m=(\ell -m)/\rho (q_\ell )$  and $k=1/n$ we have,
  \[ \sum_{m:m\neq \ell} {\frac{k^2}{\sin^2  {\frac{k(q_\ell -q_m)}{2}} }}
      \rightarrow\sum_{m:m\neq \ell} {\frac {4\rho (q_\ell )^2}{(\ell-m)^2}}
      ={\frac{4\pi^2\rho (q_\ell )^2}{3}}  \]
as $n\rightarrow\infty$. We can also take $p_\ell=\partial q_\ell /\partial x$. This suggests the following definition.

 For $\gamma >1$ and $\kappa_U\in {\mathbb R}$, let $\rho\in L^\gamma\cap {\mathcal P}$ and $q$ be absolutely continuous such that $\partial q/\partial x\in L^{2\gamma /(\gamma -1)}$. Then let $K$ be the Hamiltonian
   \begin{equation}\label{HJK}
   K(q, \rho )
      ={\frac{1}{2}}\int_{\mathbb T} \rho (x)\left( {\frac{\partial q}{\partial x}}\right)^2\, dx
           +{\frac{\kappa_U}{\gamma (\gamma -1)}}\int_{\mathbb T} \rho (x)^\gamma\, dx
   \end{equation}
with canonical variables the functions $q, \rho: {\mathbb T}\rightarrow {\mathbb R}$, and Poisson bracket
   \[ \int \{ F, G\}(x)\, dx
    =\int\!\!\!\int \left( {\frac{\partial F}{\partial q}}(x){\frac{\partial G}{\partial\rho}}(y)
         -{\frac{\partial F}{\partial\rho}}(x){\frac{\partial G}{\partial q}}(y)\right) \delta_0(x-y)\, dxdy,
   \]
where $\delta_0$ denotes the unit point mass at $0$. The thermodynamic variables are the velocity and density $(v,\rho )$ where $v=\partial q/\partial x$.

We interpret ${\frac{1}{2}}\int_{\mathbb T} \rho (x)v(x)^2\, dx$ as the kinetic energy. The Hamilton--Jacobi equation is
   \[
   {\frac{\partial S}{\partial t}}+{\frac{1}{2}}\int_{\mathbb T} \rho (x)\left( {\frac{\partial }{\partial x}}
                       {\frac{\partial S }{\partial \rho }}\right)^2\, dx
    + {\frac{\kappa_U}{\gamma (\gamma -1)}}\int_{\mathbb T} \rho (x)^\gamma \, dx
    =0.
   \]
The canonical equations are
    \begin{equation}\label{Eul}
   {\frac{\partial }{\partial t}}\begin{bmatrix} \rho \\ v\end{bmatrix} +\begin{bmatrix} v&\rho \\ \kappa_U \rho^{\gamma -2}&v\end{bmatrix}  {\frac{\partial }{\partial x}}\begin{bmatrix} \rho \\ v\end{bmatrix} =0,
   \end{equation}
with initial data $v(0,x)$ and $\rho (0,x)$, which
one recognises as Euler's equations for isentropic dynamics of a compressible gas, as in (\ref{conservation}) and (\ref{limitpole}). The state of the isentropic gas is completely determined by $\rho$ and $v$, and thermodynamic entropy is constant in space and time. The top entry in the column vector  is the conservation law, and the bottom entry is the equation of motion. Here
the eigenvalues of the matrix are $\lambda_{\pm}=v\mp \sqrt{\kappa_U} \rho^{(\gamma-1)/2}$, and the Riemann invariants are
    \[
       r_{\pm}=v\mp {\frac{2\sqrt{\kappa_U}}{\gamma -1}}\rho^{(\gamma-1)/2}.
    \]
In particular the choice $\gamma =3$ corresponds to an isentropic gas with one degree of freedom, and we recover the equations (\ref{conservation}) and (\ref{limitpole}). Taking the laws of thermodynamics into account, Selinger and Whitham \cite{SelWhith} show that the Lagrangian corresponding to $K$ with constraints is $\int\rho^3$. We consider this quantity below. LeFloch and Westdickenberg \cite{LeFW} obtained a global existence theorem for solutions of the one-dimensional isentropic Euler equations under the hypotheses of finite mass and finite total energy.

First, we concentrate attention upon the density.

\begin{prop} Let $\rho_0,\rho_1\in {\mathcal P}\cap L^3$. Then there exists a Lipschitz continuous path $\hat \rho_t$ in $({\mathcal P}_2, W_2)$
from $\rho_0$ to $\rho_1$, and a measurable family of functions $v_t\in L^2(\hat\rho_t)$ with antiderivatives $q_t=\int^x v_t$, that satisfy the continuity equation
    \begin{equation}\label{continuityequation}
    {\frac{\partial\hat \rho_t}{\partial t}}+{\frac{\partial}{\partial x}}\bigl( v_t(x)\hat \rho_t (x)\bigr)=0,
    \end{equation}
and such that the mean of the Hamiltonian satisfies
   \[
   \int_0^1 K(\hat \rho_t, q_t) \, dt
      \leq {\frac{1}{2}}W_2^2(\rho_0, \rho_1)+{\frac{1}{2}}\bigl( U(\rho_0) +U(\rho_1)\bigr).
   \]
\end{prop}

\begin{proof} As in Theorem 13.8 of \cite{V2009}, Brenier showed that there exists a Lipschitz continuous path $\hat \rho_t$ in ${\mathcal P}_2$ and a measurable family of functions $v_t\in L^2(\rho_t)$ such that (\ref{continuityequation}) holds. Then
   \[
      W_2(\rho_0, \rho_1)=\inf_v\Bigl\{ \Bigl( \int_0^1 \int_{[0, 2\pi ]} \vert v_t(x)\vert^2\rho_t (dx)dt\Bigr)^{1/2}\Bigr\}
   \]
where the infimum is taken over all paths in ${\mathcal P}_2$ and measurable families of functions $v_t\in L^2(\rho_t)$; the minimising family $v_t\in L^2(\rho_t)$ is essentially unique.
Thus $(1/2)W_2(\rho_0, \rho_1)^2$ is the minimum kinetic energy of a path that transports $\rho_0$ to $\rho_1$.
We let $\varphi_1(x)$ be the monotonic function that satisfies $\int_0^{\varphi_1(x)}\rho_1(u)\, du=\int_0^x\rho_0(u)\, du$, and generally introduce $\varphi_t(x)=(1-t)x+t\varphi_1(x)$ with ${\frac{\partial}{\partial t}} \varphi_t(x)=\varphi_1(x)-x$; so we define $v(x)=\varphi_1(x)-x$ to be the constant velocity along the geodesic, and we obtain a family of probability density functions $(\hat \rho_t)_{0\leq t\leq 1}$ such that $\int_0^{\varphi_t(x)}\hat \rho_t(u)\, du=\int_0^x \rho_0(u)\, du$ and
   \[ W_2(\rho_0, \rho_1)=\int v(x)^2\rho_0(x)\, dx.  \]

Suppose that there exists $\delta>0$ such that $\rho_1(x)>\delta $ for all $x$. Then $\varphi_1(x)$ is absolutely continuous and $\rho_1(\varphi_1(x))\varphi_1'(x)=\rho_0(x)$ almost surely.

In particular, when $\hat\rho_t$ is the probability density function induced from $\rho_0$ by  the function $\varphi_t$, we have $\int f(x)\hat \rho_t(x)dx=\int f(\varphi_t(x))\rho_0(x)dx$ for all continuously differentiable functions $f$. Differentiating this identity, we obtain the continuity equation with $v(x,t)={\frac{\partial\varphi_t}{\partial t}}(y)$ where $\varphi_t(y)=x$. Hence we have
  \begin{align*}
    \rho (\varphi_t (x),t)\, {\frac{\partial \varphi_t }{\partial x}}(x) & =\rho (x,0), \\
   {\frac{\partial \varphi_t}{\partial t}}(x) & =v(\varphi_t (x),t),
  \end{align*}
and the kinetic energy satisfies
   \[
   {\frac{1}{2}}\int \vert v(x,t)\vert^2 \hat \rho (x,t)\, dx
      ={\frac{1}{2}}\int \vert v(\varphi_t(y),t)\vert^2\rho_0(y)\, dy
      ={\frac{1}{2}}\int \left\vert{\frac{\partial\varphi_t}{\partial t}}(y)\right\vert^2\rho_0(y)\, dy.
   \]

We introduce the internal energy density,
  \[ U_0(\rho )={\frac{2\pi^2}{3}}\rho^3 \]
such that $U(\rho )=\int_{\mathbb T} U_0(\rho (\theta)) \,d\theta$ is the internal energy for a fluid with density $\rho$. Here 
\begin{itemize}
\item[(i)] $U_0(\lambda )/\lambda \rightarrow 0$ as $\lambda\rightarrow 0+$;
\item[(ii)] $\lambda\mapsto \lambda U_0(1/\lambda )$ is strictly convex and decreasing on $(0, \infty )$; 
\item[(iii)] $U_0(\lambda+\mu)\leq 4(U_0(\lambda )+U_0(\mu ))$\ for all $\lambda, \mu >0$;
\end{itemize}
hence $U$ satisfies the conditions of Section 10.4.3 of \cite{AGS} for an internal energy. Also, $U$ is regarded as physically realistic in gas dynamics \cite{GW}. By results of McCann discussed in \cite{V2003},
the internal energy is displacement convex in the sense that
    \[
    U(\hat \rho_t)\leq (1-t)U(\rho_0)+tU(\rho_1)\qquad (t\in [0, 1]),
    \]
hence the mean value of $U(\hat\rho_t)$ is less than or equal to the average of $U(\rho_0)$ and $U(\rho_1)$. The Hessian is computed
by formula (15.7) of \cite{V2009}, but in our case, we have a simple formula. Indeed, we have
   \begin{align}\label{dispconv}
   {\frac{d^2}{dt^2}}U(\hat\rho_t)
      &={\frac{d^2}{dt^2}}{\frac{2\pi^2}{3}}\int \hat\rho_t(x)^3\, dx\nonumber\\
      &={\frac{d^2}{dt^2}}{\frac{2\pi^2}{3}}\int {\frac{\rho_0 (x)^3}{\varphi'_t(x)^2}}\, dx\nonumber\\
      &={4\pi^2}\int \rho_0(x)^3 {\frac{ (\varphi_1'(x)-1)^2}{(t\varphi_1'(x)+1-t)^4}}\, dx\geq 0.
   \end{align}
The continuity equation may be regarded as a variational formula for $\rho_s$ when one moves along the tangent direction $v$ in the Wasserstein space.
\end{proof}

\begin{rem} (i) Along the optimal transport trajectories, we use the continuity equation to show that
   \begin{align*}
   {\frac{\partial}{\partial t}} \Sigma (\hat\rho_t\mid  \rho_0 )
     &=\int_{\mathbb T} {\mathcal L}(\hat \rho_t-\rho_0) {\frac{\partial\hat\rho_t}{\partial t}}\, dx\\
     &=-\int_{\mathbb T} {\mathcal L}(\hat\rho_t-\rho_0){ \frac{\partial}{\partial x}}\bigl( v\hat\rho_t\bigr) \, dx\\
     &=\int_{\mathbb T} {\mathcal H}(\hat\rho_t-\rho_0)\bigl( v\hat\rho_t\bigr) \, dx,
   \end{align*}
so by integrating with respect to $t$ and applying the Cauchy--Schwarz inequality, we deduce that
   \begin{align}
   \Sigma (\rho_1\mid\rho_0)
      &\leq \left(\int_0^1\int_{\mathbb T}v(x,t)^2\hat \rho_t(x)\, dxdt\right)^{1/2}
            \left(\int_0^1\int_{\mathbb T}\bigl(  {\mathcal H}(\hat\rho_t-\rho_0)\bigr)^2\hat\rho_t(x)\, dxdt\right)^{1/2}\nonumber\\
      &=W_2(\rho_1, \rho_0)\left(\int_0^1 I_F (\hat\rho_t\mid \rho_0)\, dt\right)^{1/2}  ,
    \end{align}
which is a converse to the free transportation inequality in the style of the $HWI$ inequality proven in \cite{LedPop}.

(ii) Let $\gamma =3$, so the canonical equations are
   \[ {\frac{\partial q}{\partial t}}
       +{\frac{1}{2}}\left({\frac{\partial  q}{\partial x}}\right)^2
       +{\frac{\kappa_U }{2}}\rho_t(x)^2=0,\]
and we can write this in terms of the potential by $\rho ={\mathcal H}{\frac{\partial Q}{\partial x}}$.
For comparison, if
   \begin{equation}
   {\frac{\partial q}{\partial t}}+\beta \left(  {\frac{\partial q}{\partial x}}\right)^2
       +{\mathcal H}{\frac{\partial^2 q}{\partial x^2}}=0\end{equation}
then $u= {\frac{\partial q}{\partial x}}$ satisfies the Benjamin--Ono equation (\ref{BO-eqn}). 
\end{rem}

The Hamiltonian $K$ has a quadratic form with domain in $L^2(\rho )$ which we now specify more precisely. Let
   \begin{equation*}\dot {\mathcal D}^1(\rho_0)
      =\left\{ q:{\mathbb T}\rightarrow {\mathbb R}: \text{$q$ is absolutely continuous},
        \  \int q(x)\rho_0 (x)\,dx=0,   \ {\frac{\partial q}{\partial x}}\in L^2(\rho_0 )\right\},
   \end{equation*}
with the norm
   \[ \Vert q\Vert_{\dot {\mathcal D}^1}
     =\left( \int_{\mathbb T}\left(  {\frac{\partial q}{\partial x}}\right)^2\, \rho_0 (x)\, dx\right)^{1/2}.
   \]
We introduce ${\mathcal D}^{-1}$ as the closure of $L^2({\mathbb T})$ for the norm
  \[ \Vert h\Vert_{{\mathcal D}^{-1}}
      =\sup\left\{ \int h(x) q(x) \, dx: q\in \dot {\mathcal D}^1,\  \Vert q\Vert_{\dot {\mathcal D}^1}\leq 1\right\}. \]
We regard ${\mathcal D}^{-1}$ as the tangent space to $({\mathcal P}_2({\mathbb T}), W_2)$ at $\rho_0$, and  $\dot {\mathcal D}^1$ as the cotangent space to
  $({\mathcal P}_2({\mathbb T}), W_2)$ at $\rho_0$. (See \cite[Theorem 7.26]{V2003}.) For example, given $\rho_0\in L^3\cap {\mathcal P}$, we can select $Q$ such that ${\mathcal H}{Q'}=\rho_0$, so $Q, {\mathcal H}Q\in {\mathcal D}^1(\rho_0)$ and we can take $q_0=Q$ as an initial choice of the phase so that $K(\rho_0, q_0)<\infty$.

Then we consider $\{ (\rho_0 , q):\rho_0\in {\mathcal P}\cap L^3; q\in {\mathcal  D}^1(\rho_0)\}$ and map
\begin{equation}\label{Riem} (\rho_0, q)\mapsto (r_+, r_-)= \left(  {\frac{\partial q}{\partial x}}+\sqrt{\kappa_U}\rho_0, {\frac{\partial q}{\partial x}}-\sqrt{\kappa_U}\rho_0\right)\end{equation}
so $r_{\pm}\in L^2(\rho_0 )$. One can use the Riemann invariants as new variables, and obtain simple wave solutions
\begin{equation}x= {\frac{r_-f(r_+)+r_+g(r_-)}{r_++r_-}}, \qquad t= {\frac{g(r_-)-f(r_+)}{2(r_++r_-)}}\end{equation}
in terms of arbitrary differentiable functions $f$ and $g$; see \cite[p.~248]{Sned}.

The following result is largely contained in \cite{GW} and \cite{WestW}, and included for  completeness.

\begin{prop}\label{potentialprop} Let $\gamma=3$ and $\kappa_U =4\pi^2$.
\begin{enumerate}[(i)]
\item Suppose that $\rho_0\in {\mathcal P}_2({\mathbb T})$ and $\rho_0$ is absolutely continuous with $\rho_0d\rho_0/dx\in L^2(\rho_0)$, and that $q\in {\mathcal D}^1(\rho_0)$. There exists a solution of the isentropic Euler equations with  initial conditions $(\rho_0, v_0)$ where $v_0=\partial q/\partial x$.
\item Suppose further that $\rho_0$ is the equilibrium density of a potential $Q$  that satisfies (\ref{freecurv}), and let $(\rho ,v)$ satisfy the Euler equations  (\ref{Eul}). Then
    \begin{equation*}
       \left\vert {\frac{\partial}{\partial t}} \Sigma (\rho_t\mid \rho_0 )\right\vert \leq  {\frac{9\sqrt{3}}{\pi}} K,
    \end{equation*}
and there exists $C_\kappa >0$ such that
   \begin{equation*}
   W_2(\rho_t ,\rho_0)^2 \leq C_\kappa Kt.\qquad (t\geq 0).
   \end{equation*}
\end{enumerate}
\end{prop}

\begin{proof} (i) By the Cauchy--Schwarz inequality ${\frac{d\rho_0^3}{dx}}$ is integrable, so $\rho_0^3$ is bounded and hence integrable over ${\mathbb T}$ and hence $U(\rho_0)$ is finite. The Lagrangian form of the isentropic Euler equations  is
    \begin{equation}
    {\frac{d}{dt}}
       \begin{bmatrix}
          \varphi_t(x)\\
          v_t(\varphi_t (x))
       \end{bmatrix}
    = \begin{bmatrix}
         v_t(\varphi_t(x))\\
        -\bigl(\rho_t\circ \varphi_t(x)\bigr) \bigl(({\frac{\partial \rho_t}{\partial x}})\circ \varphi_t(x) \bigr)
      \end{bmatrix},
    \end{equation}
which is a system of ordinary differential equations in $t$ for fixed $x$. Westdickenberg and
Wilkening proposed a discrete approximation to this system of ODE, based upon a backward Euler method
   \[ \begin{bmatrix} X_{n+1}\\ V_{n+1}\end{bmatrix}
      =\begin{bmatrix} X_{n}\\ V_{n}\end{bmatrix}
        +\tau  \begin{bmatrix}
             V_{n+1}\\
             -(\rho_{n\tau}\circ X_{n+1} )\bigl({\frac{\partial \rho_{n\tau}}{\partial x}}\bigr)\circ X_{n+1}
             \end{bmatrix}.\]
Let $\tau>0$, and suppose that we have a starting position $X_0$, an initial density $\rho_0$ and velocity field $V_0$, so our initial data is $(X_0, \rho_0, V_0)$. In the notation of \cite{V2003} page 4, we write $\varphi \sharp \rho_0$ for the probability density function induced from $\rho_0$ by $\varphi$. The particles move from $X_0$ to $X_0+\tau V_0$, then we select $\varphi_\tau$ to be the transportation map that minimises
   \[ E(\varphi )={\frac{1}{2\tau^2}}\int \bigl( \varphi  (x)-(x+\tau V_0(x))\bigr)^2\rho_0(x)\, dx+U (\varphi \sharp \rho_0),\]
which exists by the convexity condition (\ref{dispconv}). Observe that the identity transformation $\varphi (x)=x$ gives 
$E(\varphi )= (1/2)\int v_0(x)^2\rho_0(x)\,dx +U(\rho_0)=K$, so $E(\varphi_\tau )\leq K$. 
 The optimal choice $\varphi_\tau$ induces a new density $\rho_\tau =\varphi_\tau\sharp  \rho_0$ from $\rho_0$,  and gives new position $X_{n+1}=\varphi_\tau (X_n)$, with which one updates the velocity to
   \[ 
    V_1=V_0 -\tau \bigl( \rho_\tau \circ X_1\bigr) \left({\frac{\partial \rho_\tau }{\partial x}}\circ X_1\right).
   \]
This gives the new initial data $(X_1, \rho_\tau , V_1)$ with which we can proceed to the next step. Gangbo and Westdickenberg \cite{GW}, (4.26) and (3.3) establish the crucial estimate
   \[
   {E(\rho_\tau, v_\tau) +\tau^2 2\pi^2} \int \left( {\frac{\partial \rho_\tau}{\partial x}}\right)^2\rho_\tau^3 (x)\, dx
         \leq E(\rho_0, v_0)
   \]
to show that the kinetic energy remains finite. Furthermore, one can use such energy estimates to show that as $\tau\rightarrow 0$, there exists a curve $t\mapsto \rho_t\in {\mathcal P}_2$ absolutely continuous with respect to Wasserstein distance, and that the continuity equation holds in the sense of distributions; see \cite[Proposition 5.3]{GW}.

(ii) By the continuity equation
   \begin{align}
   {\frac{\partial}{\partial t}} \Sigma (\rho_t\mid \rho_0 )
     &=\int_{\mathbb T} {\mathcal L}(\rho_t-\rho_0) {\frac{\partial\rho_t}{\partial t}}\, dx\nonumber\\
     &=-\int_{\mathbb T} {\mathcal L}(\rho_t-\rho_0){ \frac{\partial}{\partial x}}\bigl( v\rho_t\bigr) \, dx\nonumber\\
     &=\int_{\mathbb T} {\mathcal H}(\rho_t-\rho_0)\bigl( v\rho_t\bigr) \, dx\nonumber\\
     &\leq{\frac{\alpha}{2}}\int_{\mathbb T} v^2\rho_t dx
          + {\frac{1}{2\alpha}}\int_{\mathbb T} \rho_t\bigl( {\mathcal H}(\rho_t-\rho_0)\bigr)^2 \, dx,
    \end{align}
where we choose $\alpha= 9\sqrt{3}/(2\pi)$. Then by M.~Riesz's theorem $\Vert {\mathcal H}(\rho_t-\rho_0) \Vert_{L^3}\leq 3^{3/2}\Vert \rho_t-\rho_0 \Vert_{L^3}$, so by H\"older's inequality we have
   \begin{align}
    \left\vert{\frac{\partial}{\partial t}} \Sigma (\rho_t\mid \rho_0 )\right\vert 
    & \leq {\frac{\alpha }{2}}\int_{\mathbb T} v(x,t)^2\rho (x,t)\, dx + {\frac{9^2}{2\alpha }}\int_{\mathbb T} \rho (x,t)^3 \,   
            dx+{\frac{36}{\alpha }} \int\rho (x,0)^3\, dx\nonumber\\
    &\leq\alpha K(\rho_t, q_t)+{\frac{12}{\sqrt{3}\pi}}K(\rho_0, q_0)\nonumber\\
    &\leq 2\alpha K(\rho_0, q_0),
   \end{align}
using the fact that the Hamiltonian $K$ defined at (\ref{HJK}) is autonomous, and hence invariant under the canonical flow.
It now follows from the free transportation inequality (\ref{FTI}) for $\rho_0$ that
    \[
       W_2(\rho_t, \rho_0)^2\leq C_\kappa\Sigma (\rho_t, \rho_0)\leq C_\kappa Kt.
    \]
\end{proof}

\begin{rem}(i) Although the bound on $K$ ensures that the kinetic energy is finite, the velocity $v$ could be unboundedly large on sets where $\rho$ is small.
Suppose however that $Q\in C^3$ and that there exists a $\delta>0$ such that
   \[  Q''(\theta) \geq \delta -{\frac{1}{2\log 2}}.  \]
Then by \cite[Corollary 2]{Pop}, the corresponding equilibrium density is continuous and satisfies $\rho_0(\theta) >\delta 2\log 2$ for all $\theta\in {\mathbb T}$. Popescu \cite[Theorem~4]{Pop} has also obtained a version of the free Poincar\'e inequality
   \[
     \delta^2 \int\!\!\!\int_{{\mathbb T}^2} \Bigl\vert {\frac {f(e^{i\theta})-f(e^{i\phi})}{e^{i\theta}-e^{i\phi}}}\Bigr\vert
          \, {\frac{d\theta}{2\pi}}{\frac{d\phi}{2\pi}}
     \leq \int_{\mathbb T} \Bigl\vert f'(e^{i\theta})-\int f'(e^{i\phi})\rho_0 (e^{i\phi})d\phi\Bigr\vert^2 \rho_0(\theta)\, d\theta
   \]
where $f'(e^{i\theta})={\frac{d}{d\theta}}f(e^{i\theta})$. Hence the formal inclusion map  ${\mathcal D}^1\rightarrow \dot H^{1/2}$ is bounded, so when we apply this to the difference of the potentials corresponding to $\rho_\tau$ and $\rho_0$, as in $f=Q_\tau-Q_0$, we deduce the free information inequality
   \[
     \delta^{-2} I_F(\rho_0\mid \rho_\tau)\geq \Vert\rho_0-\rho_\tau\Vert^2_{\dot H^{-1/2}}.
   \]
DiPerna \cite{DiP} discusses viscosity solutions of the isentropic Euler equations under the hypothesis $\rho_0(\theta) >\delta 2\log 2$, and concludes that cavities do not develop in finite time in a viscous gas.

(ii) Suppose that there exists $0<\theta_0<\dots <\theta_n<2\pi$ and $h\in C^\infty ({\mathbb T}; {\mathbb R})$ such that
   \[ \rho_0(\theta)=\prod_{j=0}^n \vert e^{i\theta}-e^{i\theta_j}\vert^{1/2} h(\theta ).\]
Then ${\frac{d}{d\theta}} \rho_0(\theta )^2$ determines an $L^2({\mathbb T}, {\mathbb R})$ function, as in Proposition~\ref{potentialprop}.

(iii) Matytsin \cite{Mat} considers the case of $\gamma =3$ and $\kappa_U=-1$, so that the isentropic Euler equations with negative pressure, and the wave speed in (\ref{Riem}) becomes purely imaginary. In this case, the Euler equations show that $g=\rho +iv$ is a holomorphic function of $z=x+it$, which we can regard as the complex version of $g=r_+$. Then Burgers equation becomes $i{\frac{\partial g}{\partial t}}+g{\frac{\partial g}{\partial z}}=0$. Matytsin shows that $g$ satisfies a Dirichlet boundary value problem on the strip
 $\Sigma_1=\{ z=x+it: -\infty <x<\infty,\  0\leq t\leq 1\}$.
The Burgers equation superficially resembles (\ref{odef}), so it is not surprising that some of the solution formulae look alike. Guionnet \cite{GuonnetM} develops this theme further in the context of random matrices.
\end{rem}


\begin{thebibliography}{99}
\bibitem{Ablowitz} M.~J. Ablowitz and H. Segur, {\sl Solitons and the inverse scattering transform}, (SIAM, 1981).

\bibitem{Alb} S. Albeverio, R. Hoegh-Krohn and D. Merlini, Euler flows, associated generalized random fields and Coulomb systems, pp.~216--244, in {\sl Infinite dimensional analysis and stochastic processes}, Edr. S. Albeverio, (Pitman, 1985).


\bibitem{AGS} L. Ambrosi, N. Gigli and G. Sarar\'e, {\sl Gradient flows in metric spaces and in the space of probability measures}, (Second edition) (Birkhauser, 2008). 

\bibitem{AmT} C.~J. Amick and J.~F. Toland, Uniqueness and related analytic properties for the Benjamin--Ono equation --- a nonlinear Neumann problem in the plane, {\sl Acta Math.} {\bf 167} (1991), 107--126.

\bibitem{BakEm} D. Bakry and M. {\'E}mery, Diffusions hypercontractives, {\sl S{\'e}minaire de probabiliti{\'e}s XIX}, pp.~177--206, Lecture Notes in Math. {\bf 1123} (Springer, Berlin, 1985).


%

\bibitem{Ben1967} T. Benjamin, Internal waves of permanent form in fluids of great depth,
{\sl J. Fluid Mech.} {\bf 29} (1967), 559--562.

\bibitem{Bl2001} G. Blower, Almost sure weak convergence for the generalized orthogonal ensemble,
{\sl J. Statist. Phys.} {\bf 105} (2001), 309--335.

%
%


\bibitem{Bour1994}
J. Bourgain, Periodic nonlinear Schr{\"o}dinger equation and invariant measures,
{\sl Comm. Math. Phys.} {\bf 166} (1994), 1--26.


\bibitem{BullCaud} R.~K. Bullough and P.~J. Caudry, The soliton and its history,
pp.~1--64, in {\sl Solitons}, Edrs R.K. Bullough and P.J. Caudry, (Springer, 1980).


\bibitem{Case} K.~M. Case, Properties of the Benjamin--Ono equation,
{\sl J. Math. Phys.} {\bf 20} (1979), 972--977.



\bibitem{DTV2015} Y. Deng, N. Tzvetkov and N. Visciglia, Invariant measures and long-term behaviour for the Benjamin--Ono equation III,
{\sl Comm. Math. Phys.} {\bf 339} (2015), 815--857.

\bibitem{DiP} R.~J. DiPerna, Convergence of the viscosity method for isentropic gas dynamics,
{\sl Comm. Math. Phys.} {\bf 91} (1983), 1--30.



\bibitem{GW} W. Gangbo and M. Westdickenberg, Optimal transport for the system of isentropic Euler equations,
{\sl Comm. Partial Differential Equations} {\bf 34} (2009), 1041--1073.

\bibitem{GuonnetM} A. Guionnet, First order asymptotics of matrix integrals; a rigorous approach towards the understanding of matrix models,
{\sl Comm. Math. Phys.} {\bf 244} (2004), 527--569.


\bibitem{HPU} F. Hiai, M. Petz and  Ueda,  Free transportation cost inequalities via random matrix approximation,
{\sl Probab. Theory Related Fields} {\bf 130} (2004), 199--221.

\bibitem{HS} R. Holley and D. Stroock, Logarithmic Sobolev inequalities and stochastic Ising models,
{\sl J. Stat. Phys.} {\bf 46} (1987), 1159--1194.

\bibitem{Jo} K. Johansson, On random matrices from the compact classical groups,
{\sl Ann. Math.} {\bf 145} (1997), 519--545.


\bibitem{LRS1988}
J.~L. Lebowitz, H.~A. Rose and E.~R. Speer, Statistical mechanics of the nonlinear Schr{\"o}dinger equation,
{\sl J. Statist. Phys.} {\bf 50} (1988), 657--687.



\bibitem{LedPop} M. Ledoux and I. Popescu, Mass transportation proofs of free functional inequalities, and free Poincar{\'e} inequalities,
{\sl J. Funct. Anal.} {\bf 257} (2009), 1175--1221.

\bibitem{LeFW} P.~G. LeFloch and M. Westdickenberg, Finite energy solutions of the isentropic Euler equations with geometric effects,
{\sl J. Math. Pures Appl. (9)} {\bf 88} (2007), 389--429.

%

\bibitem{Mat} A. Matytsin, On the large-$N$ limit of the Itzykson--Zuber integral,
{\sl Nuclear Phys. B} {\bf 411} (1994), 805--820.

\bibitem{Mol2007}
L. Molinet,
Global well-posedness in the energy space for the Benjamin--Ono equation on the circle,
{\sl Math. Ann.}  {\bf 337}  (2007),  353--383.

\bibitem{OPS} B. Osgood, R. Phillips and P. Sarnack, Extremals of determinants of Laplacians,
{\sl J. Funct. Anal.} {\bf 80} (1988), 148--211.

\bibitem{Pop} I. Popescu, Free functional inequalities on the circle,
{\sl Adv. Math.} {\bf 330} (2018), 1101--1159.

%

\bibitem{SelWhith} R.~L. Selinger and G.~B. Whitham, Variational principles in continuum mechanics,
{\sl Proc. Royal Soc. London A} {\bf 305} (1968), 1--25.

\bibitem {Sned} I. Sneddon, {\sl Elements of Partial Differential Equations}, (McGraw--Hill, 1957).


\bibitem{V2003} C. Villani, {\sl Topics in optimal transportation}, Graduate Studies in Mathematics, 58, (American Mathematical Society, 2003).

\bibitem{V2009}
C. Villani, {\sl Optimal transport.  Old and new}, Grundlehren der Mathematischen Wissenschaften, 338. (Springer, 2009).


\bibitem{WestW} M. Westdickenberg and J. Wilkening, Variational particle schemes for the porous medium equation and the system of isentropic Euler equations,
{\sl M2AN Math. Model. Numer. Anal.} {\bf 44} (2010), 133--166.

\bibitem{Zyg}
A. Zygmund, {\sl Trigonometric series}, 3rd Edition, (Cambridge University Press, 2003).

\end{thebibliography}
\end{document}